\documentclass[10pt]{article}
\usepackage{amsfonts,color}
\usepackage{amsmath}
 \usepackage{epsfig}
\usepackage{lscape}
\usepackage{amssymb}
\usepackage{graphicx}
\usepackage{makeidx}
\usepackage{amssymb}
 \usepackage{footnote}
\usepackage[utf8]{inputenc}
\usepackage{amsmath}	
\usepackage{mathtools}
\usepackage{amsthm,amssymb,wasysym}
\usepackage{amsfonts}
\usepackage[english]{babel}

\makeatletter
\let\@fnsymbol\@arabic
\makeatother

\definecolor{purple}{rgb}{.5,.1,.7}

\allowdisplaybreaks[1]

 \newtheorem{theorem}{Theorem}[section]
  \newtheorem{Def}[theorem]{Definition}
 \newtheorem{lemma}[theorem]{Lemma}
 \newtheorem{remark}[theorem]{Remark}
 \newtheorem{proposition}[theorem]{Proposition}
 \newtheorem{cor}[theorem]{Corollary}
 \newtheorem{conjecture}[theorem]{Conjecture}

\newcommand{\R}{\mathbb{R}}
 \newcommand{\N}{\mathbb{N}}
 \newcommand{\lam}{\lambda}
 \newcommand{\norm}[2][]{\|#2\|_{#1}}
 \newcommand{\matr}[1]{\begin{pmatrix}#1\end{pmatrix}}
 \newcommand{\tel}[1]{\frac{1}{#1}}

 \DeclareMathOperator{\Sym}{Sym}
 
 \DeclareMathOperator{\Cof}{Cof}
 \DeclareMathOperator{\dev}{dev}

\newcommand{\Rplus}{\R_+}
\newcommand{\phii}{\varphi}
\newcommand{\set}[1]{\{#1\}}

\newcommand{\Ktilde}{\widetilde{K}}
\newcommand{\nat}{\in\N}
\newcommand{\al}{\alpha}

\def\barr{\begin{array}}
\def\id{1\!\!1}
\def\tr{\textrm{tr}}

\def\dd{\displaystyle}

\def\barr{\begin{array}}
\def\earr{\end{array}}
\def\bec#1{\begin{equation}\label{#1}}
\def\becn{\begin{equation*}}
\def\endec{\end{equation}}
\def\endecn{\end{equation*}}

\newcommand{\limzn}{\lim_{z\searrow 0}}

\newcommand{\limxin}{\lim_{\xi\searrow 0}}

\newcommand{\limae}{\lim_{a\searrow 1}}
\newcommand{\limale}{\lim_{\al\searrow 1}}
\newcommand{\Rinf}{\mathbb{\overline{R}}}
\newcommand{\deltahat}{\widehat \delta}
\newcommand{\lambdahat}{\widehat \lambda}
\newcommand{\muhat}{\widehat \mu}

\begin{document}
\title{The exponentiated Hencky-logarithmic strain energy.\\ Part II:  Coercivity, planar polyconvexity and existence of minimizers}
\author{
Patrizio Neff\thanks{Corresponding author: Patrizio Neff,  \ \ Head of Lehrstuhl f\"{u}r Nichtlineare Analysis und Modellierung, Fakult\"{a}t f\"{u}r Mathematik, Universit\"{a}t Duisburg-Essen,  Thea-Leymann Str. 9, 45127 Essen, Germany, email: patrizio.neff@uni-due.de} \quad
and \quad
Johannes Lankeit\thanks{Johannes Lankeit,  \ \ Institut f\"{u}r Mathematik, Universit\"{a}t Paderborn,  Warburger Str. 100,
33098 Paderborn, Germany, email: johannes.lankeit@math.uni-paderborn.de} \quad
and \quad
Ionel-Dumitrel Ghiba\thanks{Ionel-Dumitrel Ghiba, \ \ \ \ Lehrstuhl f\"{u}r Nichtlineare Analysis und Modellierung, Fakult\"{a}t f\"{u}r Mathematik, Universit\"{a}t Duisburg-Essen, Thea-Leymann Str. 9, 45127 Essen, Germany;  Alexandru Ioan Cuza University of Ia\c si, Department of Mathematics,  Blvd. Carol I, no. 11, 700506 Ia\c si,
Romania; and  Octav Mayer Institute of Mathematics of the
Romanian Academy, Ia\c si Branch,  700505 Ia\c si, email: dumitrel.ghiba@uni-due.de, dumitrel.ghiba@uaic.ro}
  \quad
and\\  \quad Robert Martin\thanks{Robert Martin,  \ \ Lehrstuhl f\"{u}r Nichtlineare Analysis und Modellierung, Fakult\"{a}t f\"{u}r Mathematik, Universit\"{a}t Duisburg-Essen,  Thea-Leymann Str. 9, 45127 Essen, Germany, email: robert.martin@stud.uni-due.de } \quad and  \quad
David Steigmann\thanks{David Steigmann,  \ \ Professor of Mechanical Engineering, University of California at Berkeley,  Berkeley, CA 94720, USA,
e-mail: dsteigman@me.berkeley.edu}\quad }
\maketitle

\begin{abstract}

We consider a family of isotropic volumetric-isochoric decoupled strain energies
\begin{align*}
F\mapsto W_{\rm eH}(F):=\widehat{W}_{\rm eH}(U):=\left\{\begin{array}{lll}
\frac{\mu}{k}\,e^{k\,\|\dev_n\log {U}\|^2}+\frac{\kappa}{2\hat{k}}\,e^{\hat{k}\,[\tr(\log U)]^2}&\text{if}& \det\, F>0,\vspace{2mm}\\
+\infty &\text{if} &\det F\leq 0,
\end{array}\right.\quad
\end{align*}
based on the Hencky-logarithmic (true, natural) strain tensor $\log U$, where $\mu>0$ is the infinitesimal shear modulus, $\kappa=\frac{2\mu+3\lambda}{3}>0$ is the infinitesimal bulk modulus with $\lambda$ the first Lam\'{e} constant, $k,\hat{k}$ are dimensionless parameters, $F=\nabla \varphi$ is the gradient of deformation,  $U=\sqrt{F^T F}$ is the right stretch tensor and $\dev_n\log {U} =\log {U}-\frac{1}{n} \tr(\log {U})\cdot\id$
 is the deviatoric part (the projection onto the traceless tensors) of the strain tensor $\log U$. For small elastic strains the energies reduce to first order to the classical quadratic Hencky energy
 \begin{align*}
 F\mapsto W{_{\rm H}}(F):=\widehat{W}_{_{\rm H}}(U)&:={\mu}\,\|{\rm dev}_n\log U\|^2+\frac{\kappa}{2}\,[{\rm tr}(\log U)]^2,
 \end{align*}
 which is known to be not rank-one convex.

 The main result in this paper is that in plane elastostatics the energies of the family $W_{_{\rm eH}}$ are polyconvex for $k\geq \frac{1}{3}$, $\widehat{k}\geq \frac{1}{8}$, extending a previous finding on its rank-one convexity.   Our method uses  a judicious application of Steigmann's polyconvexity criteria based on the representation of the energy in terms of the principal invariants of the stretch tensor $U$. These energies also satisfy suitable growth and coercivity conditions. We formulate the equilibrium equations and we  prove the existence of minimizers by the direct methods of the calculus of variations.
\\
\vspace*{0.25cm}
\end{abstract}

\newpage

\tableofcontents

\section{Introduction}
\subsection{Motivation}
In the first part of a series of papers \cite{NeffGhibaLankeit}, we have introduced a nonlinear elastic energy based on certain invariants of the Hencky tensor $\log U$, namely $\|\dev_n\log\,U\|^2$ and $(\tr(\log U))^2$, where $F=\nabla \varphi$ is the gradient of deformation,  $U=\sqrt{F^T F}$ is the right stretch tensor, $\log U$ is the referential (Lagrangian) logarithmic strain tensor,
$
 \dev_n X=X-\frac{1}{n} \,\tr(X)\cdot\id
$
is the deviatoric part (the projection onto the traceless tensors) of the second order tensor $X\in\mathbb{R}^{n\times n}$ and $\norm{\cdot}$ is the Frobenius tensor norm  {(see Section \ref{auxnot} for other notations)}. We have shown that this exponentiated energy expression improves several features of the formulation with respect to mathematical issues regarding well-posedness. In this paper we will discuss the polyconvexity for this family. In order to set the stage, let us briefly recapitulate some useful details. The  considered exponentiated  Hencky-logarithmic strain type energies are
\begin{align}\label{the}
W_{_{\rm eH}}(F):=\widehat{W}_{_{\rm eH}}(U):&=\left\{\begin{array}{lll}
\underbrace{\frac{\mu}{k}\,e^{k\,\|\dev_n\log\,U\|^2}+\frac{\kappa}{2\,\widehat{k}}\,e^{\widehat{k}\,(\tr(\log U))^2}}_{\text{volumetric-isochoric split}}&\ \ \ \!\!\text{if}& \det\, F>0,\vspace{2mm}\\
+\infty &\ \ \ \!\! \text{if} &\det F\leq 0,
\end{array}\right.\\
&=\dd\left\{\begin{array}{lll}
\dd\frac{\mu}{k}\,e^{k\,\|\log \frac{U}{\det U^{1/n}}\|^2}+\frac{\kappa}{2\,\widehat{k}}\,e^{\widehat{k}\,(\log \det U)^2}&\text{if}& \det\, F>0,\vspace{2mm}\\
+\infty &\text{if} &\det F\leq 0,
\end{array}\right.\notag
\end{align}
where $\mu>0$ is the shear (distortional) modulus, $\kappa=\frac{2\mu+3\lambda}{3}>0$ is the bulk modulus with $\lambda$ the first Lam\'{e} constant and $k,\widehat{k}$ are dimensionless parameters.
 The immediate importance of the family \eqref{the} of free-energy  {functions} is seen by looking at  small (but not infinitesimally small) strains. Then the exponentiated Hencky energy $W_{_{\rm eH}}(\cdot)$ reduces to first order to  the classical  quadratic Hencky energy $\widehat{W}_{_{\rm H}}(U)$ based on the logarithmic strain tensor $\log U$:
 \begin{align}\label{th}
W_{_{\rm H}}(F):=\widehat{W}_{_{\rm H}}(U)&:={\mu}\,\|{\rm dev}_n\log U\|^2+\frac{\kappa}{2}\,[{\rm tr}(\log U)]^2.
\end{align}

 Our renewed interest in the Hencky energy is motivated by a recent finding that the Hencky energy (not the logarithmic strain itself) exhibits  a fundamental property. By purely differential geometric reasoning, in  forthcoming papers \cite{NeffEidelOsterbrinkMartin_Riemannianapproach,Neff_Osterbrink_Martin_hencky13,Neff_Nagatsukasa_logpolar13} (see also  {\cite{Neff_log_inequality13,LankeitNeffNakatsukasa}})  it will be shown that
\begin{align}\label{geoprop}
{\rm dist}^2_{{\rm geod}}\left((\det F)^{1/n}\cdot \id, {\rm SO}(n)\right)&={\rm dist}^2_{{\rm geod,\mathbb{R}_+\cdot \id}}\left((\det F)^{1/n}\cdot \id, \id\right)=|\log \det F|^2,\notag\\
{\rm dist}^2_{{\rm geod}}\left( \frac{F}{(\det F)^{1/n}}, {\rm SO}(n)\right)&={\rm dist}^2_{{\rm geod,{\rm SL}(n)}}\left( \frac{F}{(\det F)^{1/n}}, {\rm SO}(n)\right)=\|\dev_n \log U\|^2,
\end{align}
where ${\rm dist}_{{\rm geod}}$ is the canonical left invariant geodesic distance on the Lie group ${\rm GL}^+(n)$ and ${\rm dist}_{{\rm geod,{\rm SL}(n)}}$, ${\rm dist}_{{\rm geod,\mathbb{R}_+\cdot \id}}$ denote the corresponding geodesic distances on the Lie groups ${\rm SL}(n)$ and $\mathbb{R}_+\cdot \id$, respectively (see \cite{Neff_Osterbrink_Martin_hencky13,Neff_Nagatsukasa_logpolar13}).

In the first part \cite{NeffGhibaLankeit} we have summarized the well-known unique  features of the quadratic Hencky strain energy $W{_{_{\rm H}}}$ based exclusively on the natural strain tensor $\log U$.  The Hencky model is definitely one of the most widely used strain energies in the small elastic strain regime \cite{Hencky28a,Hencky29a,Hencky29b,Hencky31,Bruhns01,Bruhns02JE,henann2011large}.
In \cite{NeffGhibaLankeit}, however, we also pointed out that  the quadratic Hencky energy has some serious shortcomings. For example, the quadratic Hencky energy is neither rank-one convex nor does it satisfy any suitable coercivity condition.
These points being more or less well-known, it is clear that there cannot exist a general mathematical well-posedness result for the quadratic Hencky model $W{_{_{\rm H}}}$. Of course, in the vicinity of the identity, an existence proof  for small loads based on the implicit function theorem will always be possible. All in all,  the status of Hencky's quadratic energy is  put into doubt. This state of affairs, on the one hand the preferred  use of the quadratic Hencky energy and its fundamental property \eqref{geoprop}, on the other hand its mathematical shortcomings, motivated our search for a modification of Hencky's energy. Our best candidate for now is $W_{_{\rm eH}}$ defined by \eqref{the}. Up to moderate strains,  {for principal stretches $\lambda_i\in(0.7,1.4)$}, our new exponentiated Hencky formulation \eqref{the} is de facto as good as the quadratic  Hencky model $W{_{_{\rm H}}}$ and in the large strain region it improves several important features from a mathematical point of view. Moreover, some
other properties (see \cite{NeffGhibaLankeit}) such as  uniqueness in the hydrostatic loading problem \cite{Ogden83,chen1996stability}  confirm  the status of
the exponentiated  Hencky formulation as a useful energy in plane elasto-statics and give a new perspective in three dimensions. The main features that have been shown in \cite{NeffGhibaLankeit} is that the exponentiated Hencky energy \eqref{the} satisfies the LH-condition (rank-one convexity) in planar elasto-statics, i.e. for $n=2$. In this paper we aim to complete this investigation by showing that the planar elasto-static formulation is, in fact, polyconvex and satisfies a coercivity estimate which allows us to show the existence of minimizers. Unfortunately, some aspects of the three-dimensional description remain open, since the formulation is not globally rank-one convex.

\subsection{Polyconvexity}

  A very useful  constitutive requirement is Ball's fundamental polyconvexity condition \cite{Ball77,Ball78}. A free energy function $W(F)$ is called  polyconvex if and only if it is expressible in the form
$W(F) =P(F,\Cof F,\det F)$, $P:\mathbb{R}^{19}\rightarrow\mathbb{R}$, where $P(\cdot,\cdot,\cdot)$ is convex. Polyconvexity implies weak-lower semicontinuity, quasiconvexity and rank-one convexity and it implies that the homogeneous solution $\varphi(x)=\overline{F}.\, x, \ x\in\R^3$, is always an energy minimizer to its own Dirichlet boundary conditions.

In fact, polyconvexity is the cornerstone notion for a proof of the existence of minimizers by the direct methods of the calculus of variations
for energy functions satisfying no polynomial growth conditions, which is  the case in nonlinear elasticity since one has the natural requirement
$W(F)\rightarrow\infty$ as $\det F\rightarrow0$. Polyconvexity is best understood for isotropic energy functions, but it is  not restricted to isotropic response.
The polyconvexity condition in the case of space dimension 2 was conclusively discussed  by Rosakis \cite{Rosakis98} and \v{S}ilhav\'{y} \cite{Silhavy97,Silhavy99b,Silhavy02b,Silhavy03,silhavy2002monotonicity,vsilhavy2001rank,SilhavyPRE99}, while the case of arbitrary spatial dimension was studied by Mielke \cite{Mielke05JC}.  The $n$-dimensional case of  the theorem established by Ball \cite[page 367]{Ball77} has been reconsidered by Dacorogna and Marcellini \cite{DacorognaMarcellini}, Dacorogna and Koshigoe \cite{DacorognaKoshigoe} and  Dacorogna and Marechal \cite{DacorognaMarechal}. It was a long standing open question how to extend the notion of polyconvexity in a meaningful way to anisotropic materials \cite{Ball02}. An answer has been provided in a series  of papers \cite{Neff_Vortrag_Schroeder_IUTAM02,HutterSFB02,Neff_Diss00,Balzani_Schroeder_Gross_Neff05,Schroeder_Neff_Ebbing07,cism_book_schroeder_neff09,Schroeder_Neff01,Hartmann_Neff02,Schroeder_Neff_Iutam01,
Schroeder_Neff04,Balzani_Neff_Schroeder05,Schroeder_Neff_Ebbing07,Ebbing_Schroeder_Neff_AAM09}.

\subsection{Approach of this paper}
The main result in this paper is that in plane elastostatics the family of energies $W_{_{\rm eH}}$ given by \eqref{the} is polyconvex for a suitable choice of parameters $k,\hat k$ (Theorem  \ref{mainth}), satisfies $q$-growth coercivity for any $1\leq q<\infty$, (Theorem \ref{thuncoe}) and therefore allows for a complete existence theory (Theorem \ref{mainexist}). This also confirms the status of the quadratic Hencky energy as a useful approximation in plane elasto-statics. Moreover, our family \eqref{the} of energies admits a unique, stress-free reference configuration $\id$, thus $\varphi(x)=x$ is the global minimizer for natural boundary conditions in any dimension.

The sufficiency condition for polyconvexity which we use has been discovered by Steigmann \cite{SteigmannMMS03,SteigmannQJ03}. Eventually, it is based on a polyconvexity criterion of Ball \cite{Ball77}, but it allows one to express polyconvexity directly in terms of the principal isotropic invariants of the right stretch tensor $U$, namely $i_1=\tr\, U, i_2=\det U$ (see also \cite{Davis57,Lewis03,Lewis96,Lewis96b,Borwein2010}). As  it turns out, in plane elastostatics, Steigmann's criterion is already
hidden in another sufficiency criterion for polyconvexity given earlier by Rosakis \cite{Rosakis92}. However,  Steigmann's criterion is clearly not necessary for polyconvexity (see Section \ref{acpse}).

\subsection{Notation}\label{auxnot}
Let us begin with the remark, that although this article is mainly
concerned with the planar (two-dimensional) case, we give some of the
preliminaries in their more general three-dimensional version.
 For $a,b\in\R^n$ we let $\langle {a},{b}\rangle_{\R^n}$  denote the scalar product on $\R^n$ with the
associated vector norm $\|a\|_{\R^n}^2=\langle {a},{a}\rangle_{\R^n}$.
We denote by $\R^{n\times n}$ the set of real $n\times n$ second order tensors, written with
capital letters.
The standard Euclidean scalar product on $\R^{n\times n}$ is given by
$\langle {X},{Y}\rangle_{\R^{n\times n}}=\tr({X Y^T})$, and thus the Frobenius tensor norm is
$\|{X}\|^2=\langle {X},{X}\rangle_{\R^{n\times n}}$. In the following we do not adopt any summing convention and we omit the subscript
$\R^n,\R^{n\times n}$. The identity tensor on $\R^{n\times n}$ will be denoted by $\id$, so that
$\tr{(X)}=\langle {X},{\id}\rangle$. We let $\Sym(n)$ and $\rm PSym(n)$ denote the sets of symmetric and positive definite symmetric tensors respectively and adopt the usual abbreviations of Lie-group theory, i.e.
${\rm GL}(n):=\{X\in\R^{n\times n}\;|\det{X}\neq 0\}$ is the general linear group,
${\rm SL}(n):=\{X\in {\rm GL}(n)\;|\det{X}=1\},$\,  ${\rm GL}^+(n):=\{X\in\R^{n\times n}\;|\det{X}>0\}$   is the group of invertible matrices with positive determinant. The superscript $^T$ is used to denote transposition, and $\Cof A = (\det A)A^{-T}$ is the cofactor of $A\in {\rm GL}^+(3)$. The set of positive real numbers is
denoted by $\R_+:=(0,\infty)$, while $\overline{\R}_+:=\R_+\cup \{\infty\}$.

Let $\Omega\subset{\R^n}$  be a bounded domain with Lipschitz boundary
$\partial\Omega$. Let us consider $W(F)$ to be the strain energy  density function of an elastic material in which $F$ is the deformation gradient  from a reference configuration to a configuration in Euclidean $n$-space; $W(F)$ is measured per unit volume of the reference configuration. The domain of $W(\cdot)$ is ${\rm GL}^+(n)$.  We denote by $C=F^T F$ the right Cauchy-Green strain tensor, by $B=F\, F^T$ the left Cauchy-Green (or Finger) strain tensor, by $U$ the right stretch tensor, i.e. the unique element of ${\rm PSym}(n)$ for which $U^2=C$, and by  $V$ the left stretch tensor, i.e. the unique element of ${\rm PSym}(n)$ for which $V^2=B$. Here,  we are only concerned with rotationally symmetric energy functions (objective and isotropic), i.e.
$
 W(F)=\widehat{W}(Q_1^T\, F\, Q_2)$ for all $F=R\,U=V R\in {\rm GL}^+(n),\  Q_1,Q_2,R\in{\rm SO}(n).
$
For vectors $v=\left(v_1,v_2,v_3\right)^T\ \in\R^3,
$ we define
$
{\rm diag}\, v=\left(
                          \begin{array}{ccc}
                            v_1 & 0 & 0 \\
                            0 & v_2 & 0 \\
                            0 & 0 & v_3 \\
                          \end{array}
                        \right),
$
while for a matrix
$
F=\left(
                          \begin{array}{ccc}
                            F_{11} & F_{12} & F_{13} \\
                            F_{21} & F_{22} & F_{23} \\
                            F_{31} & F_{32} & F_{33} \\
                          \end{array}
                        \right)\in\R^{3\times 3}
$ we let
$
{\rm vect}\, F=(F_{11}, F_{12}, F_{13}, F_{21}, F_{22},F_{23}, F_{31}, F_{32}, F_{33})^T\in \R^9.
$

If the components of the $\R^3$-valued  vector  field $v=\left(v_1,v_2,v_3\right)^T$ are  differentiable in the distributional sense, we define
\begin{align}\label{defgrad}
 {\nabla }\,v=\left(
  \begin{array}{c}
   {\rm grad}^T\,  v_1 \\
    {\rm grad}^T\, v_2 \\
    {\rm grad}^T\, v_3 \\
  \end{array}
\right)\, ,
\end{align}
while for a weakly differentiable scalar function $(x_1,x_2,x_3)\mapsto f(x_1,x_2,x_3)\in \R$ the gradient is the column vector
\begin{align}\label{defgrad1}
\nabla f:={\rm grad} f=\left(\frac{\partial f}{\partial x_1},\frac{\partial f}{\partial x_2},\frac{\partial f}{\partial x_3}\right)^T\in \R^3.
\end{align}

In three dimensions, we consider the singular values (principal stretches) $\lambda_1$, $\lambda_2$, $\lambda_3$ of $F$, i.e. the eigenvalues  of  $U$, and the principal isotropic invariants of $U$
\begin{align}\label{ifl}
i_1&=\lambda_1+\lambda_2+\lambda_3=\tr(U)\,, \notag\\
i_2&=\lambda_1\lambda_2+\lambda_2\lambda_3+\lambda_3\lambda_1=\tr(\Cof U)\,,\\
i_3&=\lambda_1\lambda_2\lambda_3=\det U\,.\notag
\end{align}
 Every isotropic and frame-invariant function of $F$ is thus expressible in the form
\begin{align*}
W(F) &=\widehat{W}(U)=g(\lambda_1, \lambda_2, \lambda_3)=\psi(i_1,i_2,i_3)=\Phi(\lambda_1, \lambda_2, \lambda_3,\lambda_1 \lambda_2,\lambda_2 \lambda_3, \lambda_3 \lambda_1,\lambda_1 \lambda_2 \lambda_3)=P(F,\Cof F, \det F).
\end{align*}
The functions $\widehat{W}, g,\psi$ are uniquely determined by $W$, while $\Phi$ and $P$ are not unique.

We denote by $D^2_{\lambda}g$  the Hessian matrix of $g$ with respect to the variables $(\lambda_1,\lambda_2,\lambda_3)$, while by $D^2_i \psi$ we denote the Hessian matrix of $\psi$ with respect to the principal invariants  $(i_1,i_2,i_3)$. We also consider the  third order tensor
$
\mathbb{D}^2_\lambda i=(D^2_\lambda i_1|D^2_\lambda i_2|D^2_\lambda i_3),
$
where $D^2_\lambda i_1$, $D^2_\lambda i_2$ and $D^2_\lambda i_3$ denote the Hessian matrices of $i_1, i_2, i_3$ with respect to $\lambda$.

\section{Preliminary results}\setcounter{equation}{0}
\subsection{The  sum of squared logarithms
inequality}

In this paper we also use the
sum of squared logarithms inequality   recently demonstrated  in \cite{Neff_log_inequality13}:
\begin{theorem}\label{Noging}{\rm (The  sum of squared logarithms
inequality in 3D \cite{Neff_log_inequality13})}\\
Let $\lam_1,\lam_2,\lam_3,\mu_1,\mu_2,\mu_3\in\Rplus$ be such that
\begin{align}\label{sumofsquaredlog_condition}
 \lam_1+\lam_2+\lam_3&\leq \mu_1+\mu_2+\mu_3\,\notag,\\
  \lam_1\,\lam_2+\lam_1\,\lam_3+\lam_2\,\lam_3&\leq \mu_1\,\mu_2+\mu_1\,\mu_3+\mu_2\,\mu_3\,,\\
 \lam_1\,\lam_2\,\lam_3&=\mu_1\,\mu_2\,\mu_3\,\notag.
\end{align}
Then the following inequality holds:
\begin{align}
 \log^2\lam_1+\log^2\lam_2+\log^2\lam_3\leq \log^2\mu_1+\log^2\mu_2+\log^2\mu_3.
\end{align}
\end{theorem}

\begin{theorem}\label{Noging2}{\rm (The  sum of squared logarithms
inequality in 2D \cite{Neff_log_inequality13})}
Let $\lam_1,\lam_2,\mu_1,\mu_2\in\Rplus$ be such that
$
 \lam_1+\lam_2\leq \mu_1+\mu_2\,,\
 \lam_1\,\lam_2=\mu_1\,\mu_2\,.
$
Then the following inequality holds:
$
 \log^2\lam_1+\log^2\lam_2\leq \log^2\mu_1+\log^2\mu_2.
$
\end{theorem}
For the general $n$-dimensional case, we consider the elementary symmetric polynomials
\begin{align*}
e_k(X_1,X_2,...,X_n)=\sum\limits_{1\leq j_1<j_2<...<j_k\leq n}X_{j_1}X_{j_2}...X_{j_k},\qquad k=1,...,n
\end{align*}
and we give the conjecture:
\begin{conjecture}{\rm (The  sum of squared logarithms
inequality in $\R_+^n$, $n\in\N$)}\label{conjecture_sumofsquaredlog}
Let $\lam_1,\lam_2,...,\lam_n,\mu_1,\mu_2,...,\mu_n\in\Rplus$ be such that
\begin{align*}
 e_k(\lam_1,\lam_2,...,\lam_n)&\leq e_k(\mu_1,\mu_2,...,\mu_n), \quad  k=1,...,n-1,\\
 e_n(\lam_1,\lam_2,...,\lam_n)&= e_n(\mu_1,\mu_2,...,\mu_n).\notag
\end{align*}
Then the following inequality holds
$$
 \sum\limits_{k=1}^n \log^2\lam_k\leq \sum\limits_{k=1}^n \log^2\mu_k.
$$
\end{conjecture}

In the next section we outline the polyconvexity criterion established by Steigmann \cite{SteigmannQJ03} in terms of the principal invariants $(i_1,i_2,i_3)$ of the right stretch tensor $U$. Using  Steigmann's criterion and the criterion given by  Lemma \ref{lem}, we are able to prove the polyconvexity  of the exponentiated Hencky energy in plane finite elastostatics.

\subsection{Sufficiency criteria for polyconvex strain energies}\label{acpse}

A function $W(F)$ is polyconvex if and
only if it is expressible in the form
$W(F) =P(F,\Cof F, \det F)$, where $P(\cdot,\cdot,\cdot)$ is convex.  The notion of polyconvexity has been introduced into the framework of
elasticity by  John Ball in his seminal paper \cite{Ball77}. Various nonlinear issues, results and extensive references are collected in Dacorogna \cite{Dacorogna08}. In general, a function $\Phi(\lambda_1,\lambda_2, \lambda_3, \lambda_1\lambda_2, \lambda_2\lambda_3,\lambda_3\lambda_1,\lambda_1\lambda_2 \lambda_3)$ is polyconvex if it is convex, symmetric and monotone increasing (separately) in its first $6$ arguments, see Theorem \ref{critBalleng3}. However, it is known that the monotonicity in the first $6$ arguments is not necessary \cite{Mielke05JC}. Since there is no easy way to represent the energy in terms of $(F,\Cof F, \det F)$, we take the detour of the invariant representation. From \cite{SteigmannQJ03} we have the following result based on the interesting observation that
 the invariants $
i_1=\tr (U),\  i_2=\tr(\Cof U),\  i_3=\det U
$
 are convex with respect to $F$, $\Cof F$ and $\det F$, respectively (see \cite{SteigmannQJ03}, page 485).

\begin{proposition}{\rm (Steigman's polyconvexity criterion in 3D)}\label{sc1}
Suppose that
\begin{itemize}
\item[i)] $\psi(i_1,i_2,i_3)$ is a convex function of $(i_1,i_2,i_3)$ jointly\footnote{The domain in which $\psi(i_1,i_2,i_3)$ is defined is the domain for which $\lambda_1>0,\, \lambda_2>0,\, \lambda_3>0$, i.e. the equation $\lambda^3-i_1\lambda^2+i_2 \lambda-i_3=0$ has three positive real solutions. But this domain is not convex.  Therefore, it would be more adequate to say that $\psi(i_1,i_2,i_3)$ is convex  in the sense of Busemann, Ewald and Shephard's definition \cite{Busemann}, i.e. $\psi$   can be extended to a convex function defined on the convex hull of its domain of definition.}, and
\item[ii)] $\psi(i_1,i_2,i_3)$ is a non-decreasing function\footnote{If $\psi$ is differentiable, then this condition means that $\partial _{i_1}\psi(i_1,i_2,i_3)\geq 0,\partial _{i_2}\psi(i_1,i_2,i_3)\geq 0$ for all $(i_1,i_2,i_3)\in\mathbb{R}^3$ for which  the equation $\lambda^3-i_1\lambda^2+i_2 \lambda-i_3=0$ has three positive real solutions. } of $i_1$ and $i_2$, separately.
\end{itemize}
Then $W(F)=\psi(i_1,i_2,i_3)$ is polyconvex.
\end{proposition}

 In planar elasticity, $U\in \mathbb{R}^{2\times 2}$ and the relevant isotropic principal invariants  are
\begin{align}\label{lii}
i_1&=\lambda_1+\lambda_2=\tr(U), \qquad
i_2=\lambda_1\lambda_2=\det U.
\end{align}
We have to remark that $i_2$ from \eqref{lii} does not coincide with $i_2$ from the three dimensional case. However, it can be understood from the context which expression for $i_2$ is used. For planar elasticity  we have the corresponding result \cite{SteigmannQJ03}:

\begin{proposition}{\rm (Steigman's polyconvexity criterion in 2D)} \label{sc2}
Suppose that
\begin{itemize}
\item[i)] $\psi(i_1,i_2)$ is a convex function of $(i_1,i_2)$ jointly\footnote{The domain in which $\psi(i_1,i_2)$ is defined is the domain ${D(i_1,i_2)}$ defined in \eqref{defdomains}, for which $\lambda_1,\lambda_2>0$, which is not a convex set. Again, a more appropriate notion of convexity for the function $\psi(i_1,i_2)$ on ${D(i_1,i_2)}$ is that of Busemann, Ewald and Shephard \cite{Busemann}, i.e. that $\psi$  is the restriction to ${D(i_1,i_2)}$ of a real-valued convex
function (in the usual sense) defined on the convex hull of ${D(i_1,i_2)}$ or,  equivalently, that the function $\psi$ can be extended to a convex function defined on the convex hull $Co {D(i_1,i_2)}=\mathbb{R}_+^2$ of ${D(i_1,i_2)}$.}, and
\item[ii)] $\psi(i_1,i_2)$ is a non-decreasing function of $i_1$.
\end{itemize}
Then  $W(F)=\psi(i_1,i_2)$ is polyconvex.
\end{proposition}

Templet and Steigmann's recent claim \cite{TempletSteigman13},  that these conditions are  also necessary for polyconvexity can be easily misinterpreted. Below we present some counterexamples to this point. In fact, formula (41) in \cite{TempletSteigman13} does not take care of the possibility that e.g. the dependence of $\Phi$ on $F$ does not have to be transmitted by $i_1$ alone. For the 3D-case, Steigmann showed that the above criterion may be applied to the energy
\begin{align}
W(F)&=a^+(i_1-3)+b^+(i_2-3)+h(i_3)=a^+\langle U-\id,\id\rangle+b^+\langle \Cof U-\id,\id\rangle+h(\det F)\\
\notag
&=a^+(\lambda_1+\lambda_2+\lambda_3)+b^+(\lambda_1\lambda_2+\lambda_2\lambda_3+\lambda_3\lambda_1)+h(\det F)-3(a^++b^+),
\end{align}
where $h$ is a convex function and $a^+, b^+>0$. The polyconvexity of this energy can also be deduced from a direct application of Ball's theorem \cite{Ball77}.

Steigmann's polyconvexity criterion in the planar case  \cite{SteigmannQJ03} is already contained in  the paper by Rosakis and Simpson \cite{Rosakis92}  for the choice of the entry parameter $\alpha=-1$. Indeed, Rosakis and Simpson gave sufficient conditions for polyconvexity of $W(\cdot)$ having the form
$
W(F)=\widetilde{W}(\tr(F^TF),\det F)=\widetilde{W}(\|F\|^2,\det F).
$
In the notation of Rosakis and Simpson
$
I:=\tr(F^T F)=\|F\|^2,  J:=\det F$, $
\mathcal{A}_\alpha=\{(\xi,\eta): \xi\geq 0,  \eta\geq 0,\,  \xi^2\geq 2\,(1-\alpha)\eta\,\}.
$
Let us give the correlations with our notations. Rosakis and Simpson defined the function $\xi:{\rm GL}(2)\rightarrow\mathbb{R}$ by
$
\xi_\alpha(F)=\sqrt{\|F\|^2-2\,\alpha\,\det F},\quad F\in {\rm GL}(2),
$
and proved (see Lemma 3.1 in \cite{Rosakis92}) that the function $\xi_\alpha$ is convex\footnote{First, in \cite{Rosakis92}, it is proved that for $M\in{\rm PSym}(n)$, the function $\varphi:\R^n\to \R$, defined by $\varphi(x)=\sqrt{\langle x,M.\,x\rangle}, \ x\in\R^n$ is  convex. For $F\in {\rm GL}(2)$ we have  $\|F\|^2\geq 2\, |\det F|$ because $\|F\|^2- 2\, |\det F|=(F_{11}+F_{22})^2+(F_{12}-F_{21})^2$. Hence, the expression $ \|F\|^2-2\alpha\det F$ under the radix is, for $\alpha\in[-1,1]$, a quadratic, positive semi-definite function of $F\in {\rm GL}(2)$. By means of an isomorphism $F\mapsto {\rm vec}(F):=(F_{11},F_{12}, F_{21}, F_{22})\in \R^4$, the function $\xi_\alpha$ can be expressed as a function of the
form $\varphi_\alpha:\R^4\to \R$, defined by $\varphi_\alpha({\rm vec}(F))=\sqrt{\langle {\rm vec}(F),M_\alpha.\,{\rm vec}(F)\rangle}, \ F\in {\rm GL}(2)$, where $M_\alpha=B_\alpha^T\, B_\alpha\in{\rm PSym}(4)$ is a positive definite matrix. Thus, $\varphi_\alpha({\rm vec}(F))=\|B_\alpha {\rm vec}(F)\|, \ F\in {\rm GL}(2)$ and therefore $F\mapsto \varphi_\alpha ({\rm vec}(F))$ is convex.}
 for  $\alpha\in[-1,1]$. Moreover they  pointed out that
$
F\in{\rm GL}^+(2)\Leftrightarrow(\xi_\alpha,J)\in \mathcal{A}_\alpha.
$
Let us remark that $J=i_2$ in general and for $\alpha=-1$ the domain $\mathcal{A}_\alpha$ is the
domain $\overline{D(i_1,i_2)}$, where ${D(i_1,i_2)}$ is the domain considered in our further analysis, see \eqref{defdomains}, and $\xi_{-1}=i_1$. The convex hull of $\mathcal{A}_\alpha$ is $\mathbb{R}_+^2$ for $-1\leq \alpha<1$, while for $\alpha= 1$,
$\mathcal{A}_1=[0,\infty)\times \mathbb{R}_+$ is convex and is exactly the domain considered in our extension (see \eqref{defpsifin}).

\begin{proposition}{\rm (Rosakis and Simpson's early polyconvexity criterion in 2D \cite{Rosakis92}) }
Let $W :{\rm GL}^+(2)\mapsto \R$ be isotropic. For each $\alpha\in[-1,1]$ define
$
\Phi_\alpha(\xi,J)=\widetilde{W}(\xi^2+2\,\alpha \,J,J),\ (\xi,J)\in \mathcal{A}_\alpha,
$
and suppose that for some $\alpha\in[-1,1]$,
\begin{itemize}
\item[i)]  $\Phi_\alpha$ is convex by extension\footnote{The function $\Phi_\alpha$ is well defined in $\mathcal{A}_\alpha$ which is not equal to $\mathcal{A}_1=[0,\infty)\times \mathbb{R}_+^2$ in general.} to $\mathcal{A}_1=[0,\infty)\times \mathbb{R}_+$,
\item[ii)] $\Phi_\alpha(\cdot,J)$ is nondecreasing on $[0,\infty)$ for each $J>0$.
\end{itemize}
Then $W(\cdot)$ is polyconvex.
\end{proposition}

Rosakis and Simpson \cite{Rosakis92} already stated that the conditions of the above proposition are not necessary for polyconvexity of isotropic functions. They illustrated this with an example due to Dacorogna et al. \cite{DacorognaPRSE90}. For a complete view we give this example in the following. The considered function is  ${W}:{\rm GL}^+(2)\rightarrow\mathbb{R}$ given by
$
W(F)=\|F\|^4-2(\det F)^2=I^2-2\,J^2,  (I,J)\in \mathcal{D},
$
where
$
\mathcal{D}=\{(I,J):\, I\geq 2\,|J|,  J\in\mathbb{R}\}.
$
Then ${W}:\R^{2\times 2}\rightarrow\mathbb{R}$, ${W}(F)=I^2-2J^2$ is convex \cite{DacorognaPRSE90} and hence its restriction to ${\rm GL}^+(2)$ is polyconvex. In this case we have
$
\Phi_\alpha(\xi,J)=(\xi^2+2\,\alpha\, J)^2-2\,J^2,\ \forall\, (\xi,J)\in \mathcal{A}_\alpha.
$
The Hessian matrix of $\Phi_\alpha$ fails to be positive semi-definite on $\mathcal{A}_\alpha$ for all $\alpha\in[-1,1].$ Hence, the conditions of Rosakis and Simpson are not necessary. In our notation, we have
$
{W}(F)=\langle C,\id\rangle^2-2(\det F)^2=\langle F^T F,\id\rangle^2-2(\det F)^2=\| F\|^4-2(\det F)^2,
$
which is in fact convex in $F\in\mathbb{R}^{2\times 2}$, while in terms of principal invariants the function
$
{W}(F)=\psi(i_1,i_2)=(i_1^2-2\,i_2)^2-2\,i_2^2=i_1^4-4\,i_1^2\,i_2+2\,i_2^2,
$
does not have  a positive semi-definite Hessian on $\mathcal{A}_{-1}={D(i_1,i_2)}$.

Another counterexample for this phenomenon, but in the three-dimensional case, is  given by the mapping $F\mapsto \norm{\Cof F}^2=\norm{\Cof U}^2$, which is (obviously) {polyconvex} (since it is convex in $\Cof F$). This function is rotationally invariant.
The eigenvalues of $\Cof U$ are $\lam_2\lam_3,  \lam_1\lam_3, \lam_1\lam_2 $, hence
$
 \norm{\Cof U}^2=(\lam_1\lam_2+\lam_1\lam_3+\lam_2\lam_3)^2-2(\lam_1+\lam_2+\lam_3)(\lam_1\lam_2\lam_3)=i_2^2-2\,i_1\, i_3.\notag
$
The corresponding unique representation function $\phii(i_1,i_2,i_3)=i_2^2-2i_1i_3$ is a nonconvex function in $(i_1,i_2,i_3)$ and not even convex in a neighbourhood of $\id$, i.e. of $(i_1,i_2,i_3)=(3,3,1)$.
Although $\norm{\Cof F}^2$ is a function of $\Cof F$ only, its representation as a function of the principal invariants of $U$ also contains $i_1$ and $i_3$.
Furthermore, the resulting function is neither convex in $(i_1,i_2,i_3)$ nor increasing in $i_1$.

We give in the following some immediate consequence of the Theorems \ref{sc1} and \ref{sc2}.
\begin{remark}
 \label{thm:expsteigmann}
 If some function $\psi$ satisfies the hypotheses of Theorems \ref{sc1} or \ref{sc2}, then $e^\psi$ satisfies them as well.
\end{remark}
\begin{proof}
The proof follows from the monotonicity and convexity of the  exponential function.
\end{proof}
\begin{remark}
 \label{thm:expsteigmannb}
 If $\psi$ fails the hypothesis ii) from  Theorems \ref{sc1} or \ref{sc2}, then so does $e^ \psi$.
\end{remark}
\begin{proof} The function
 $\log$ is monotone. Hence, if $e^\psi$ were monotone in $i_1$ (or $i_2$), $\log(e^\psi)=\psi$ would be monotone in $i_1$ (or $i_2$).
\end{proof}
Note that, however, the convexity condition (i)  from  Theorems \ref{sc1} or \ref{sc2} can be improved by the exponential function. Polyconvexity is compatible with exponentiating:

\begin{remark}
 \label{thm:exppoly}
 If a function $W$ is polyconvex, then so is $e^ W$.
\end{remark}
\begin{proof} According with the definition, $W$ is polyconvex if and only if
 $W(F)=P(F,\Cof F,\det F)$, where $P$ is convex. But if  $P$ is convex, then $e^P$ is also convex (the exponential function is convex and monotone), hence $e^W$ is polyconvex \cite{Hartmann_Neff02,Schroeder_Neff01}.
\end{proof}

\subsection{Plane elastostatics}

In planar elasticity the relevant isotropic principal invariants  are defined by \eqref{lii}.
Note again that the meaning of  the isotropic invariants of $U$, namely $i_1,i_2$, depends on the dimension. Every isotropic and frame-invariant function of $F\in{\rm GL}^+(2)$ is  expressible in the form
\begin{align}\label{WgpsiPhiP}
W(F) &=\widehat{W}(U)=g(\lambda_1, \lambda_2)=\psi(i_1,i_2)=\Phi(\lambda_1, \lambda_2, \lambda_1 \lambda_2,\lambda_1 \lambda_2)=P(F, \det F),
\end{align}
where the functions $\widehat{W}, g,\psi$ are uniquely determined by $W$, while $\Phi$ and $P$ are not unique.
The same observation as in Lemma \ref{lem-martin} leads to
\begin{lemma}\label{lem-martin2}
Let $\Psi:\R_+^2\to\R$ and $\Phi:\R_+^3\to\R$ with
\begin{align}\label{lem-marti2-cond}
	\Phi(\lambda_1,\lambda_2,\delta) = \Psi(\lambda_1+\lambda_2,\:\delta)
\end{align}
for all $\lambda_1,\lambda_2,\delta\in \R_+$. Then $\Phi$ is convex if and only if $\Psi$ is convex.
\end{lemma}
\begin{remark}
Let us consider an isotropic energy function $W(\lambda_1,\lambda_2)$ with
$
	W(\lambda_1,\lambda_2) = \Phi(\lambda_1,\,\lambda_2,\lambda_1\lambda_2)
= \Psi(\lambda_1+\lambda_2,\lambda_1\lambda_2)
$
for all $\lambda_1,\lambda_2\in\R^+$ with functions $\Psi:\R_+^2\to\R$ and $\Phi:\R_+^3\to\R$. Then the functions $\Psi$ and $\Phi$ do not necessarily fulfil the conditions of the previous lemma, i.e. we do not have a strong equality like \eqref{lem-marti2-cond}, cf. Remark \ref{remA2}.
\end{remark}

We also have the 2D version of the Ball's sufficient criterion for polyconvexity of isotropic functions:

\begin{theorem}{\rm(Ball \cite[page 367]{Ball77} 2D sufficient conditions for polyconvexity of isotropic functions)}\label{critBalleng2}\\
Let
$
W(F)=\Phi(\lambda_1,\lambda_2,\lambda_1\lambda_2),
$
where $\lambda_1,\lambda_2$ are the singular values of $F\in {\rm GL}^+(2)$, and
\begin{itemize}
\item[i)] $\Phi:\mathbb{R}_+^3\rightarrow\mathbb{R}$ is convex,
\item[ii)] $\Phi(x_1,x_2,\delta)=\Phi(x_2,x_1,\delta)$ for all $x_1,x_2,\delta\in\mathbb{R}_+$,
\item[iii)] $\Phi(x_1,x_2,\delta)$ is nondecreasing in $x_1$ and $x_2$ individually.
\end{itemize}
Then $W$ is polyconvex.
\end{theorem}

Given the eigenvalues $\lambda_1, \lambda_2\in \mathbb{R}_+$, we can always compute the invariants $i_1, i_2\in\mathbb{R}$. But for given $i_1, i_2\in\mathbb{R}_+$ we can not always say that the equation
$
\lambda^2-i_2\,\lambda+i_1=0
$
has two different positive solutions $\lambda_1, \lambda_2$, which is the case if and only if $i_2^2-4\,i_1>0$. Our intention is to make the map $(\lambda_1,\lambda_2)\mapsto (i_1,i_2)$ a one-to-one function. For this reason, we define the function
\begin{align}\label{deffuncti}
i=(i_1,i_2)^T:{D(\lambda_1,\lambda_2)}\rightarrow {D(i_1,i_2)},
\end{align}
 which maps $(\lambda_1,\lambda_2)$ into $(i_1,i_2)$, where (see Figures \ref{figure1}, \ref{figure2})
\begin{align}
\label{defdomains}
{D(\lambda_1,\lambda_2)}&=\{(\lambda_1,\lambda_2)\in\mathbb{R}^2_+:\; \lambda_1> \lambda_2\},\quad
{D(i_1,i_2)}=\{(i_1,i_2)\in\mathbb{R}^2_+:\; i_1^2-4\,i_2> 0\}.
\end{align}
\begin{figure}[h!]
\hspace*{1cm}
\begin{minipage}[h]{0.4\linewidth}
\centering
\includegraphics[scale=0.6]{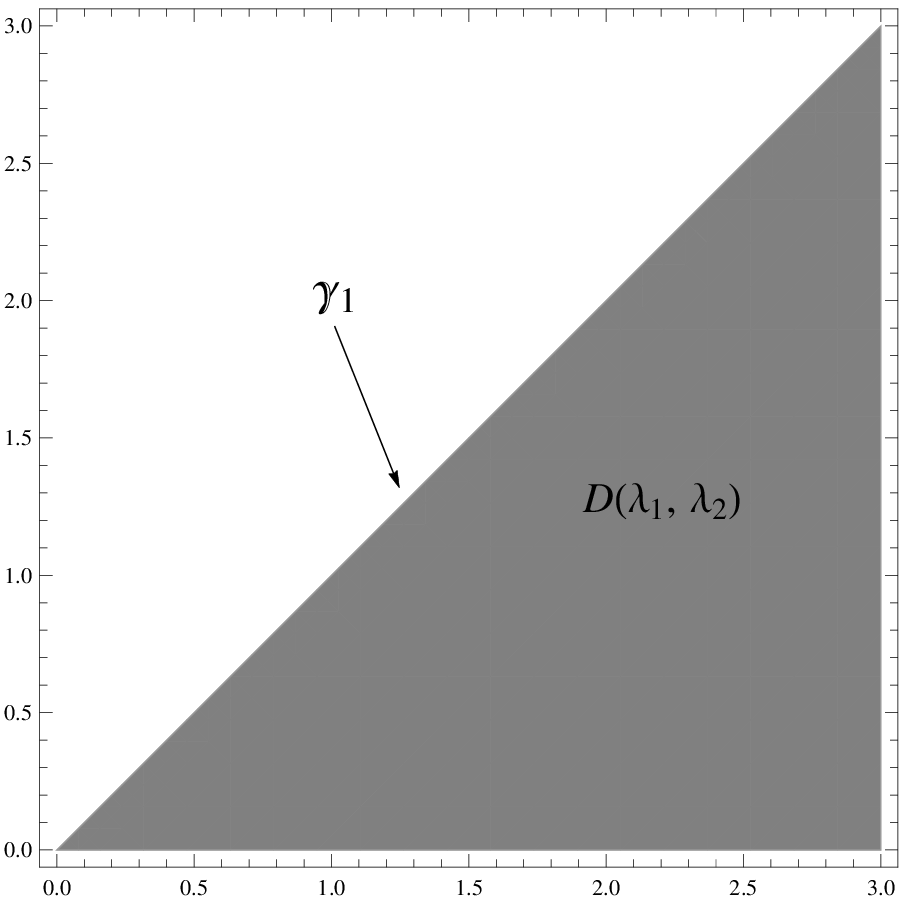}
\centering
\caption{\footnotesize{The domain ${D(\lambda_1,\lambda_2)}$ of the singular values $\lambda_1,\lambda_2$.}}
\label{figure1}
\end{minipage}
\hspace{0.5cm}
\begin{minipage}[h]{0.4\linewidth}
\centering
\includegraphics[scale=0.6]{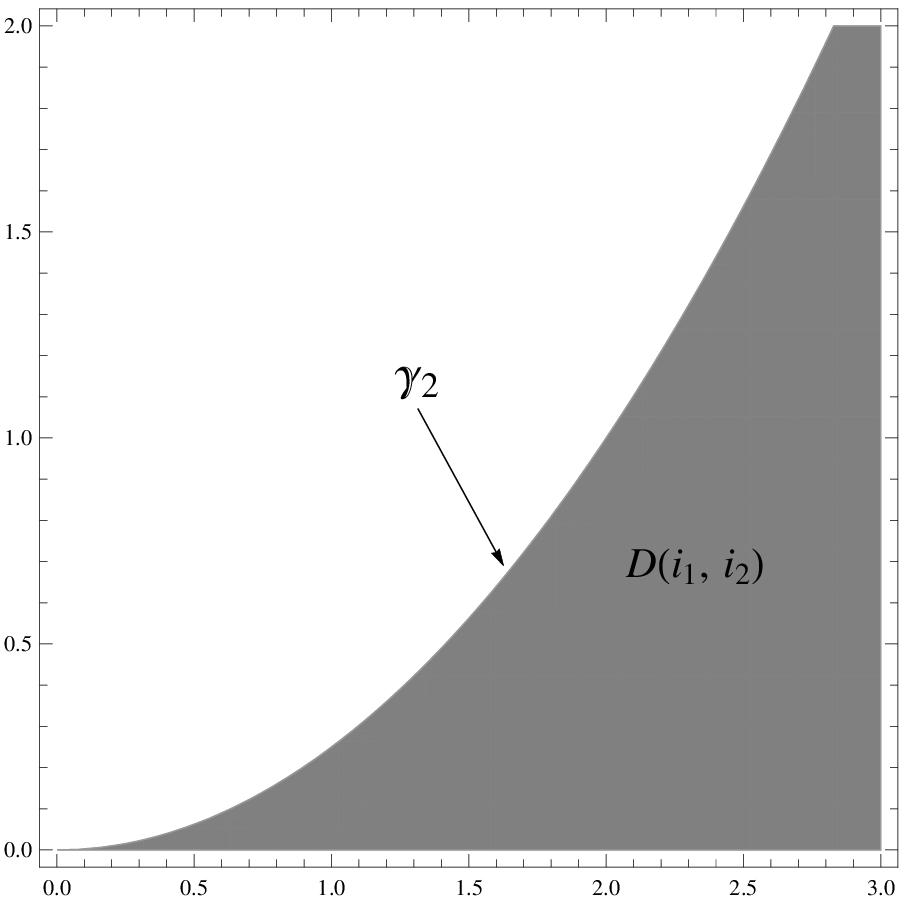}
\centering
\caption{\footnotesize{The domain ${D(i_1,i_2)}$ of the admissible values of the isotropic principal invariants $i_1,i_2$.}}
\label{figure2}
\end{minipage}

\end{figure}%

We can also define the function $i(\cdot)$ on the curve $\gamma_1:\ \lambda_1=t, \ \lambda_2=t,\  t\in(0,\infty)$. In this way $i(\cdot)$ maps the curve $\gamma_1$ into the curve $\gamma_2: i_1=t,\ i_2=\frac{t^2}{4},\  t\in(0,\infty)$. The function $i(\cdot)$ is also a one-to-one function on this curve but it is a $C^2$--diffeomorphism on the open domain $D(\lambda_1,\lambda_2)$, away from the curve $\gamma_1$. Hence, for now we consider the restriction of the function $g$ from \eqref{WgpsiPhiP} to the domain ${D(\lambda_1,\lambda_2)}$, denoted in the following also by $g$. According to \eqref{lii} and \eqref{wf}, the energy $W(F)$ can also be written in terms of $(i_1,i_2)$, i.e. there is $\psi:{D(i_1,i_2)}\rightarrow\mathbb{R}$ such that $\psi=g\circ i^{-1}$.  We suppose that the function $g$ is a $C^2$-function on its domain of definition. Moreover, the function $i=(i_1,i_2)$ is a $C^2$-diffeomorphism. Using the above notations, the chain rule and according with \eqref{defgrad} and
\eqref{defgrad1} we write
\begin{align}\label{grad2}
\nabla_\lambda g=(\nabla_\lambda i)^T\nabla_i \psi.
\end{align}
We know that for $(\lambda_1,\lambda_2)\in {D(\lambda_1,\lambda_2)}$ the matrix
$
\notag\nabla_\lambda i\dd=
\left(
\begin{array}{cc}
 1 & 1 \vspace{1mm}\\
 \lambda_2 & \lambda_1 \\
\end{array}
\right)
$
is invertible and
\begin{align}
(\nabla_\lambda i)^{-T}=\frac{1}{\lambda_1-\lambda_2}\left(
\begin{array}{cc}
 \lambda_1 & -\lambda_2 \vspace{1mm}\\
 -1 & 1 \\
\end{array}
\right).
\end{align}
In this case, we have
\begin{align}\label{hes2}
\frac{\partial^2 f}{\partial \lambda_i\partial \lambda_j}&=\frac{\partial }{\partial \lambda_j}\left(\frac{\partial \psi}{\partial i_1}\frac{\partial i_1}{\partial \lambda_i}
+\frac{\partial \psi}{\partial i_2}\frac{\partial i_2}{\partial \lambda_i}\right)\\\notag
&=\frac{\partial^2 \psi}{\partial i_1^2}\frac{\partial i_1}{\partial \lambda_j}\frac{\partial i_1}{\partial \lambda_i}+
\frac{\partial^2 \psi}{\partial i_1\partial i_2}\frac{\partial i_2}{\partial \lambda_j}\frac{\partial i_1}{\partial \lambda_i}+
\frac{\partial^2 \psi}{\partial i_1\partial i_2}\frac{\partial i_2}{\partial \lambda_i}\frac{\partial i_1}{\partial \lambda_j}+
\frac{\partial^2 \psi}{\partial i_2^2}\frac{\partial i_2}{\partial \lambda_j}\frac{\partial i_2}{\partial \lambda_i}+
\frac{\partial \psi}{\partial i_1}\frac{\partial^2 i_1}{\partial \lambda_j\partial \lambda_i}
+
\frac{\partial \psi}{\partial i_2}\frac{\partial^2 i_2}{\partial \lambda_j\partial \lambda_i}.
\end{align}
 In 2D, the Hessian matrix of $g$ with respect to the variables $(\lambda_1,\lambda_2)$ and the Hessian matrix of $\psi$ with respect to the variables $(i_1,i_2)$ are
\begin{align}
D^2_{\lambda}g=\left(
\begin{array}{ccc}
 \frac{\partial ^2 g}{\partial \lambda_1^2} & \frac{\partial ^2 g}{\partial \lambda_1\partial\lambda_2} \\
 \frac{\partial ^2 g}{\partial \lambda_1\partial\lambda_2} & \frac{\partial ^2 g}{\partial \lambda_2^2}
\end{array}
\right),\ \ \ \
D^2_{i}\psi=\left(
\begin{array}{ccc}
 \frac{\partial ^2 \psi}{\partial i_1^2} & \frac{\partial ^2 \psi}{\partial i_1\partial i_2} \\
 \frac{\partial ^2 \psi}{\partial i_1\partial i_2} & \frac{\partial ^2 \psi}{\partial i_2^2}
\end{array}
\right),
\end{align}
while
$
\mathbb{D}^2_\lambda i=(D^2_\lambda i_1|D^2_\lambda i_2).
$ We recall that for a   third  order tensor $\mathbb{G}$, we have
$
\mathbb{G}.\, Y\in \mathbb{R}^{2\times 2}$  for all $Y\in \mathbb{R}^{2},$ where $ \ (\mathbb{G}.\, Y)_{ij}=\sum\limits_{k=1}^{2}\mathbb{G}_{kij}Y_{k}\,.
$
Thus, we can write
$
\mathbb{D}^2_\lambda i.\nabla_i \psi=\sum_{j=1}^2 \frac{\partial \psi}{\partial i_j}D^2_\lambda i_j.
$ Using the above notations, we can rewrite \eqref{hes2} in the form
$
D^2_\lambda g=(\nabla_\lambda i)^TD_i^2\psi\,(\nabla_\lambda i)+\mathbb{D}^2_\lambda i.\nabla_i \psi.
$
Moreover, using \eqref{grad2} and the fact that $i:{D(\lambda_1,\lambda_2)}\rightarrow {D(i_1,i_2)}$ is a $C^2$-diffeomorphism, i.e. $\det\nabla_\lambda i\neq 0$, we deduce
$
D^2_\lambda g=(\nabla_\lambda i)^TD_i^2\psi\, (\nabla_\lambda i)+\mathbb{D}^2_\lambda i.[(\nabla_\lambda i)^{-T}\nabla_\lambda g],
$
and further
$
(\nabla_\lambda i)^TD_i^2\psi\,(\nabla_\lambda i)=D^2_\lambda g-\mathbb{D}^2_\lambda i.[(\nabla_\lambda i)^{-T}\nabla_\lambda g].
$
In view of the relation
\begin{align}
\langle (\nabla_\lambda i)^TD_i^2\psi\,(\nabla_\lambda i)\xi,\xi\rangle=\langle D_i^2\psi\,(\nabla_\lambda i)\xi,(\nabla_\lambda i)\xi\rangle,\ \ \ \forall\,\xi\in \mathbb{R}^2,
\end{align}
it is clear that
\begin{align}
\qquad D_i^2\psi\ \  &\text{is positive definite in}\ \ (i_1,i_2)\in {D(i_1,i_2)} \\
&\Leftrightarrow  D^2_\lambda g-\mathbb{D}^2_\lambda i.[(\nabla_\lambda i)^{-T}\nabla_\lambda g] \ \ \ \text{is positive definite in} \ \ (\lambda_1,\lambda_2)\in {D(\lambda_1,\lambda_2)}.\notag
\end{align}

Hence, we can conclude:
\begin{lemma}\label{lem} Let $i=(i_1,i_2)^T:{D(\lambda_1,\lambda_2)}\subset\mathbb{R}^2\rightarrow {D(i_1,i_2)}\ \subset \mathbb{R}^2$ be  the
$C^2$-diffeomorphism defined by \eqref{deffuncti},  $\psi:{D(i_1,i_2)}\rightarrow\mathbb{R}$  and $g:{D(\lambda_1,\lambda_2)}\rightarrow \mathbb{R}$  functions of class $C^2$ on their domain of definition, such that
  $g(\lambda_1,\lambda_2):=(\psi \circ i)(\lambda_1,\lambda_2)$.  Then
$D_i^2\psi$ is positive definite in ${D(i_1,i_2)}$ (as a function of $(i_1,i_2)$) if and only if $D^2_\lambda g-\mathbb{D}^2_\lambda i.[(\nabla_\lambda i)^{-T}\nabla_\lambda g]$
is positive definite in ${D(\lambda_1,\lambda_2)}.
$
\end{lemma}
It is clear that the above lemma holds true in general, for all $C^2$-diffeomorphisms $i=(i_1,i_2)^T:{D(\lambda_1,\lambda_2)}\subset\mathbb{R}^2\rightarrow {D(i_1,i_2)}\ \subset \mathbb{R}^2$.

\section{Polyconvexity of the  exponentiated Hencky energy in plane elastostatics}
\setcounter{equation}{0}
\subsection{Polyconvexity of the isochoric  exponentiated Hencky energy in plane elastostatics}\label{poly-isochoric}

In this section we consider  a variant of the exponentiated Hencky energy in plane strain, with isochoric part
\begin{align}\label{wf}
W_{\rm iso}(F)=e^{k\,\|{\rm dev}_2 \log U\|^2}=e^{k\,\|\log \frac{U}{\det U^{1/2}}\|^2}.
\end{align}
Let us remark again that for small strains the exponentiated Hencky energy reduces to the well-known quadratic Hencky energy:
\begin{align}\label{th22}
W_{_{\rm eH}}(F)-\left(\frac{\mu}{k}+\frac{\kappa}{2\widehat{k}}\right)&=\underbrace{\frac{\mu}{k}\,e^{k\,\|{\rm dev}_n\log U\|^2}+\frac{\kappa}{2\widehat{k}}\,e^{\widehat{k}\,[{\rm tr}(\log U)]^2}}_{\text{fully nonlinear elasticity}}-\left(\frac{\mu}{k}+\frac{\kappa}{2\widehat{k}}\right)\notag\\
&=\underbrace{\mu\,\|\,{\rm dev}_n\log U\|^2+\frac{\kappa}{2}\,[(\log \det U)]^2}_{\text{materially linear, geometrically nonlinear elasticity}}+\,\text{h.o.t.}\\
&=\underbrace{\mu\,\|\,{\rm dev}_n\,{\rm sym} \nabla u\|^2+\frac{\kappa}{2}\,[\tr({\rm sym} \nabla u)]^2}_{\text{linear elasticity}}+\,\text{h.o.t.}\,,\notag
\end{align}
where  $u:\mathbb{R}^n\rightarrow\mathbb{R}^n$ is the displacement,
$
F=\nabla \varphi=\id +\nabla u
$ is the gradient of deformation $\varphi:\mathbb{R}^n\rightarrow\mathbb{R}^n$ and h.o.t. denotes terms of higher order of $\|\,{\rm dev}_n\log U\|^2$ and $\frac{\kappa}{2}\,[(\log \det U)]^2$.

Coming back to the 2D case, as $W$ is an objective, isotropic tensor function, we can express it as a function of the singular values of $F$, that is the eigenvalues $\lam_1,\lam_2$ of $U=\sqrt{F^TF}$, or of the principal invariants $i_1=\lambda_1+\lambda_2$, $i_2=\lambda_1\lambda_2$, i.e.
$
W(F)=P(F, \det F)=g(\lambda_1,\lambda_2)=\psi(i_1,i_2).
$
For polyconvexity, the representation of the function $W(F)$ in terms of $P(F,\det F)$ is  not unique, see  \eqref{wf}. However, the representations $W(F)=g(\lambda_1,\lambda_2)
 =\psi(i_1,i_2)$ are unique. This fact is implied by the following lemma:
\begin{lemma}
\label{thm:w2d}
 Let $k\in\R$ and the matrix $F\in{\rm GL}^+(2)$ with singular values $\lam_1,\lam_2$.  Then
 \begin{align}\label{wf3}
W(F)=e^{k\,\|{\rm dev}_2 \log U\|^2}=e^{k\,\|\log \frac{U}{\det U^{1/2}}\|^2}=g(\lambda_1,\lambda_2),\ \ \text{where}\  g:\mathbb{R}^2_+\rightarrow\mathbb{R},\ \  g(\lambda_1,\lambda_2):=e^{\frac{k}{2}\,\left(\log \frac{\lambda_1}{\lambda_2}\right)^2}.
\end{align}
\end{lemma}
\begin{proof} The matrix $U$ is positive definite and symmetric and therefore can be assumed, by the spectral representation, to be diagonal, to obtain
\begin{align*}
 \norm{\dev_2\log U}^2&=\norm{\log U-\tel2(\log \lam_1+\log \lam_2)\id}^2=\left\|\matr{\tel2\log\lam_1-\tel2\log\lam_2&0\\0&\tel2\log\lam_2-\tel2\log\lam_1}\right\|^2\\
 &=\tel4\left[2\,(\log \lam_1-\log\lam_2)^2\right]=\tel2\left(\log\frac{\lam_1}{\lam_2}\right)^2. \qedhere
\end{align*}
\end{proof}
\begin{remark}{\rm (Non-convexity of $g$)}
 Note that  the function $g:\mathbb{R}_+^2\rightarrow\mathbb{R}$ defined in  \eqref{wf3} is not convex. We have for the Hessian
\begin{align}\label{hesf} D^2_\lambda g=k\,e^{\frac{k}{2} \log ^2\frac{\lambda_1}{\lambda_2}}
\left(
\begin{array}{cc}
 \frac{k \log ^2\frac{\lambda_1}{\lambda_2}}{\lambda_1^2}-\frac{  \log \frac{\lambda_1}{\lambda_2}}{\lambda_1^2}+\frac{1 }{\lambda_1^2} & -\frac{k \log ^2\frac{\lambda_1}{\lambda_2}}{\lambda_1 \lambda_2}-\frac{1 }{\lambda_1 \lambda_2} \\
 -\frac{k \log ^2\frac{\lambda_1}{\lambda_2}}{\lambda_1 \lambda_2}-\frac{1 }{\lambda_1 \lambda_2} & \frac{k \log ^2\frac{\lambda_1}{\lambda_2}}{\lambda_2^2}+\frac{ \log \frac{\lambda_1}{\lambda_2}}{\lambda_2^2}+\frac{1}{\lambda_2^2} \\
\end{array}
\right)
\end{align}
and, for all $k\in\R$,
$
{\rm det} \, D^2_\lambda g=-\frac{k^2 \log ^2\frac{\lambda_1}{\lambda_2} \,e^{k \log ^2\frac{\lambda_1}{\lambda_2}}}{\lambda_1^2 \lambda_2^2}\leq 0,$ for all $\lambda_1,\lambda_2>0.
$
 The function $g$ is, therefore,  not a convex function in $\lambda_1,\lambda_2$.
By a general theorem \cite[Theorem 5.1]{Ball77}, this implies that $e^{k\,\|{\rm dev}_2 \log U\|^2}$ is not convex as a function of $U\in{\rm {\rm Psym}}(2)$. Thus the non-convexity of $g$ allows us to conclude that $W$ cannot be a convex function of $F$ \cite{Davis57}.
\end{remark}

In the following, we can assume without loss of generality (by the symmetry of $\log^2 \frac{\lambda_1}{\lambda_2}$ under inversion), that
$
\lambda_1\geq \lambda_2.
$

\begin{proposition}\label{prop:posdefHessian}
The map
\begin{align}\label{defpsi}
\psi:{D(i_1,i_2)}\rightarrow \mathbb{R}_+,\quad \psi(i_1,i_2)=e^{\frac{k}{2}\log^2 \frac{i_1+\sqrt{i_1^2-4i_2}}{i_1-\sqrt{i_1^2-4i_2}}}
\end{align}
 has a positive definite Hessian matrix $D^2\psi $ in the domain ${D(i_1,i_2)}$, as a function of $(i_1,i_2)$, if and only if  $k\geq \dd\frac{1}{3}$.
\end{proposition}
\begin{proof}  To prove this result we will use the criterion given by Lemma \ref{lem}. Let us remark that
\begin{align}\label{suplim}
\mathbb{D}^2_\lambda i.[(\nabla_\lambda i)^{-T}\nabla_\lambda g]=- k\,e^{\frac{k}{2} \log ^2\frac{\lambda_1}{\lambda_2}}  \left(\log \frac{\lambda_1}{\lambda_2}\right)\frac{1}{(\lambda_1-\lambda_2)}\left(\frac{1}{\lambda_1}+\frac{1}{ \lambda_2}\right)\left(
\begin{array}{cc}
 0 & 1 \\
 1 & 0 \\
\end{array}
\right)
\end{align}
and
\begin{align}
{\rm det}\, ( \mathbb{D}^2_\lambda i.[(\nabla_\lambda i)^{-T}\nabla_\lambda g])=k^2 e^{k \log ^2\frac{\lambda_1}{\lambda_2}}  \left(\log^2 \frac{\lambda_1}{\lambda_2}\right)\frac{1}{(\lambda_1-\lambda_2)^2}\left(\frac{1}{\lambda_1}+\frac{1}{ \lambda_2}\right)^2\geq 0.
\end{align}
In order to justify the above relations we outline the following calculations:
\begin{align}
&D^2_\lambda i_1=\left(
\begin{array}{cc}
 0 & 0 \\
 0 & 0 \\
\end{array}
\right)
,\ \ \ \
D^2_\lambda i_2=
\left(
\begin{array}{cc}
 0 & 1 \\
 1 & 0 \\
\end{array}
\right),\qquad
\nabla_\lambda g= k\, e^{\frac{k}{2} \log ^2\frac{\lambda_1}{\lambda_2}}  \log \frac{\lambda_1}{\lambda_2}
\left(
\begin{array}{c}
 \frac{1}{\lambda_1} \vspace{1mm}\\
 -\frac{1}{\lambda_2} \\
\end{array}
\right),
\vspace{2mm}\\
&\notag\nabla_\lambda i\dd=
\left(
\begin{array}{cc}
 1 & 1 \vspace{1mm}\\
 \lambda_2 & \lambda_1 \\
\end{array}
\right),
\ \
(\nabla_\lambda i)^{-T}=\frac{1}{\lambda_1-\lambda_2}\left(
\begin{array}{cc}
 \lambda_1 & -\lambda_2 \vspace{1mm}\\
 -1 & 1 \\
\end{array}
\right),
\vspace{3mm}
\\\notag
&(\nabla_\lambda i)^{-T}\nabla_\lambda g=k \,e^{\frac{k}{2} \log ^2\frac{\lambda_1}{\lambda_2}}  \left(\log \frac{\lambda_1}{\lambda_2}\right)\frac{1}{(\lambda_1-\lambda_2)}\left(
\begin{array}{c}
 2 \vspace{1mm}\\
 -\frac{1}{\lambda_1}-\frac{1}{ \lambda_2} \\
\end{array}
\right).
\end{align}

In view of \eqref{hesf} and \eqref{suplim} we have
\begin{align}
&D^2_\lambda g-\mathbb{D}^2_\lambda i.[(\nabla_\lambda i)^{-T}\nabla_\lambda g]\\&=k\, e^{\frac{k}{2} \log ^2\frac{\lambda_1}{\lambda_2}}\left(
\begin{array}{cc}
 \frac{ k \log ^2\frac{\lambda_1}{\lambda_2}-\log \frac{\lambda_1}{\lambda_2}+1}{\lambda_1^2} & -\frac{k (\lambda_1-\lambda_2) \log ^2\frac{\lambda_1}{\lambda_2}-(\lambda_1+\lambda_2) \log \frac{\lambda_1}{\lambda_2}+\lambda_1-\lambda_2}{\lambda_1 (\lambda_1-\lambda_2) \lambda_2} \\
 -\frac{k (\lambda_1-\lambda_2) \log ^2\frac{\lambda_1}{\lambda_2}-(\lambda_1+\lambda_2) \log \frac{\lambda_1}{\lambda_2}+\lambda_1-\lambda_2}{\lambda_1 (\lambda_1-\lambda_2) \lambda_2} & \frac{ k  \log ^2\frac{\lambda_1}{\lambda_2}+\log \frac{\lambda_1}{\lambda_2}+1}{\lambda_2^2} \\
\end{array}
\right).\notag
\end{align}
First, let us study the sign of the   (1,1)-entry $
 \widetilde{g}(\lambda_1,\lambda_2):= e^{\frac{k}{2} \log
^2\frac{\lambda_1}{\lambda_2}} \frac{1}{\lambda_1^2}\left[\,k \log
^2\frac{\lambda_1}{\lambda_2}-  \log
\frac{\lambda_1}{\lambda_2}+1\right]$ of the above matrix, which is related to the Hessian matrix of $\psi(i_1,i_2)$.
We introduce the function $r:[0,\infty)\rightarrow\mathbb{R}$ given by
$
 r(t)= k\, t^2-  t+1.
$
It is clear that if $k>\dd\frac{1}{4}$, then
$
 r(t)= k\, t^2-  t+1>\left(\frac{1}{2}t-1\right)^2\geq 0, $ for all $t\in\mathbb{R}.$
Moreover, if $r(t)>0$ for all $t\in [0,\infty)$, then
$
k>\frac{1}{4}=\max\limits_{t\in[0,\infty)}\left\{\frac{t-1}{t^2}\right\}.
$
Thus, $r(t)>0$ for all $t\in[0,\infty)$ if and only if\,\footnote{In fact, $k\, t^2-
 t+1>0$ for all $t\in\mathbb{R}$ if and only if $k>\dd\frac{1}{4}$\,.}
$k>\dd\frac{1}{4}$. In consequence, we deduce
\begin{align}\label{devsepconv}
 &\widetilde{g}(\lambda_1,\lambda_2)=k \,e^{\frac{k}{2} \log
^2\frac{\lambda_1}{\lambda_2}}
\frac{1}{\lambda_1^2}\,r\left(\log\frac{\lambda_1}{\lambda_2}
\right)>0\ \ \text{for all }\lambda_1\geq\lambda_2\in\R^+\ \ \text{if and
only if} \ \ \ k>\dd\frac{1}{4}.
\end{align}

On the other hand
\begin{align}
{\rm det}[D^2_\lambda g-\mathbb{D}^2_\lambda i.[&(\nabla_\lambda i)^{-T}\nabla_\lambda g] ]\\
&=\frac{2 k^2  e^{k \log ^2\frac{\lambda_1}{\lambda_2}} }{\lambda_1^2 \lambda_2^2 (\lambda_1-\lambda_2)^2}\left(\log \frac{\lambda_1}{\lambda_2}\right)\left[k{\left(\lambda_1^2-\lambda_2^2\right)} \log ^2\frac{\lambda_1}{\lambda_2}-{\left(\lambda_1^2+\lambda_2^2\right)} \log \frac{\lambda_1}{\lambda_2}+{\left(\lambda_1^2-\lambda_2^2\right)}\right].\notag
\end{align}
Hence
\begin{align}
{\rm det}[D^2_\lambda g&-\mathbb{D}^2_\lambda i.[(\nabla_\lambda i)^{-T}\nabla_\lambda g] ]>0\ \ \forall\, (\lambda_1,\lambda_2)\in {D(\lambda_1,\lambda_2)}\\\Leftrightarrow
&\left(\log \frac{\lambda_1}{\lambda_2}\right)\left[k {\left(\lambda_1^2-\lambda_2^2\right)} \log ^2\frac{\lambda_1}{\lambda_2}-{\left(\lambda_1^2+\lambda_2^2\right)} \log \frac{\lambda_1}{\lambda_2}+{\left(\lambda_1^2-\lambda_2^2\right)}\right]>0\ \ \forall\, (\lambda_1,\lambda_2)\in {D(\lambda_1,\lambda_2)}.\notag
\end{align}

In the following we will prove that for all $\lambda_1>\lambda_2>0$,
$
  \left(\log \frac{\lambda_1}{\lambda_2}\right)\lambda_2^2\,
\widehat{r}\left(\frac{\lambda_1}{\lambda_2}\right)>0,
$
where the function $\widehat{r}: (0,\infty)\rightarrow\mathbb{R}$ is defined by
$
\widehat{r}(t):=k\, {\left(t^2-1\right)} \log ^2t-{\left(t^2+1\right)} \log  t+{\left(t^2-1\right)}.
$
\begin{figure}[h!]
\centering
\includegraphics[scale=1]{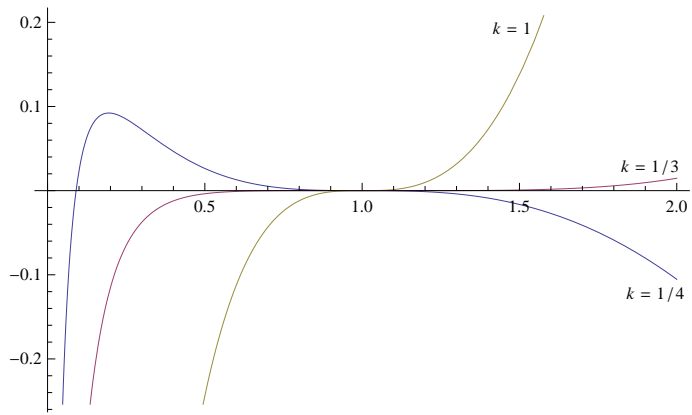}
\caption{{\footnotesize $\widehat{r}(t)$ for different values of $k$.}}
\label{figure3}
\end{figure}%
To this aim, we prove  that
$\widehat{r}(t)> 0,$ for all $t\in(1,\infty)$ if and only if $k\geq
\dd\frac{1}{3}$.

The first derivative of $\widehat{r}$ is given by
$
\widehat{r}'( t)=\frac{k}{2} \left(4\, t\, \log ^2 t+4 \,t \,\log  t-\frac{4 \,\log  t}{t}\right)+t-\frac{1}{t}-2\, t \,\log  t,
$
 the second derivative by
$
\widehat{r}''(t)=\frac{k}{2} \left(\frac{4\, \left(t^2-1\right)}{t^2}+4\, \left(\frac{1}{t^2}+3\right) \log  t+4 \,\log ^2 t\right)-\frac{t^2-1}{t^2}-2 \,\log  t,
$
and
$
\widehat{r}'''(t)=2\,\frac{ (3\,k-1) \left(t^2+1\right)+2\,k\, \left(t^2-1\right) \log  t}{t^3}.
$
We also have that
$
\widehat{r}'(1)=0 \ \ \ \text{and} \ \ \ \widehat{r}''(1)=0.
$  It is easy to see that, if
 $ k\geq\frac{1}{3}$
 then
 $
(3\,k-1) \left(t^2+1\right)+2\,k \left(t^2-1\right) \log  t> 0,$
for all $t\in (1,\infty)$.

 This means that $r''(\cdot)$ is a monotone increasing function, which implies
 $\widehat{r}''(t)> \widehat{r}''(1)=0$ if $t>1.
 $
This implies that $\widehat{r}'(\cdot)$ is 
 monotone increasing on $(1,\infty)$, i.e.
 $\widehat{r}'(t)> \widehat{r}(1)=0$ if $t>1.$
 Hence
 $
 \widehat{r}'(t)> 0 \ \ $ for all $t\in(1,\infty)$,
 i.e. $\widehat{r}$ is monotone increasing.  In conclusion, if
$k\geq\dd\frac{1}{3}$, then $\widehat{r}$ is monotone increasing and
convex on $(1,\infty)$, and $\widehat{r}'(1)=0=\widehat{r}(1)$. Hence, we have
proved that $\widehat{r}(t)>0$ for all $t\in(1,\infty)$ if  $k\geq \dd\frac{1}{3}$. In fact
$\widehat{r}'''(t)\geq 0$ for all $t\in (0,\infty)$ if and only if
\begin{align}
k\geq \frac{1}{3}=\sup\limits_{t\in (1,\infty)}\left\{\frac{t^2+1}{3 \,t^2+2\, t^2 \log  t-2\, \log  t+3}\right\}.
\end{align}
This  completes the proof.
\end{proof}

It is possible to have a
direct proof of the positive definiteness of the Hessian matrix $D^2\psi $ in
the domain ${D(i_1,i_2)}$ but this direct method leads to complicated calculations in the three-dimensional case (see Appendix \ref{convdir}).
\begin{remark}
Assuming $\lambda_1>\lambda_2$, with the help of the substitution
$t=\dd\frac{\lambda_1}{\lambda_2}$ {and the choice $k=2$ in
\eqref{wf3}} we obtain the function $s:(1,\infty)\rightarrow\mathbb{R}$,
$s(t)=e^{(\log t)^2}$. The function $s(\cdot)$ is convex and monotone increasing
in $t$, for $t\in(1,\infty)$. However,
\begin{itemize}
\item[i)] $(\lambda_1,\lambda_2)\mapsto t=\dd\frac{\lambda_1}{\lambda_2}$ is not convex  as a function of $(\lambda_1,\lambda_2)$;
\item[ii)] $(i_1,i_2)\mapsto t=\dd\frac{\lambda_1}{\lambda_2}=\dd\frac{i_1+\sqrt{i_1^2-4i_2}}{i_1-\sqrt{i_1^2-4i_2}}$ is not convex as a function of  the two invariants $(i_1,i_2)$.
\end{itemize}
It seems, therefore, that the conclusion of convexity of the map $\psi$ defined by \eqref{defpsi} cannot simply be inferred from the composition of a convex mapping with the convex and non-decreasing mapping $s:(1,\infty)\rightarrow\mathbb{R}$, $s(t)=e^{(\log t)^2}$.
\end{remark}

 In the following we prove that the function $\psi $ considered in Proposition \ref{prop:posdefHessian} is  convex on ${D(i_1,i_2)}$ in the sense of Busemann, Ewald and Shephard's definition \cite{Busemann}, i.e. $\psi$  is the restriction to ${D(i_1,i_2)}$ of a real-valued convex
function (in the usual sense) defined on the convex hull of ${D(i_1,i_2)}$; equivalently,  the function $\psi$ can be extended to a convex function defined on the convex hull $Co {D(i_1,i_2)}=\mathbb{R}_+^2$ of ${D(i_1,i_2)}$ \cite{Rosakis92}. From \cite{Busemann} we have:
\begin{theorem}{\rm (Busemann, Ewald and Shephard \cite[page
6]{Busemann})}\label{ThBES}
A function $\phi$ defined on an arbitrary set $M\subset\mathbb{R}^n$ is convex if and only if it is bounded linearly below and the inequality
\begin{align}\label{condBES}
\phi(x)\leq\sum\limits_{i=1}^r\mu_i\, \phi(x_i),\ \ 1\leq r<\infty
\end{align}
holds for all $x_1,x_2,...,x_r\in M$ , $\sum\limits_{i=1}^r\mu_i=1$ and $x=\sum\limits_{i=1}^r\mu_i x_i$ lying in $M$. The convex extension of $\phi$ to the convex hull of $M$ is
$
\widehat{\phi}(x)=\inf\left\{\sum\limits_{i=1}^r\mu_i\, \phi(x_i)\ : \ x=\sum\limits_{i=1}^r\mu_i x_i,\
\sum\limits_{i=1}^r\mu_i=1,\ 1\leq r<\infty\right\}.
$
\end{theorem}

As we already mentioned in the previous section, according to the definition
\eqref{lii}, we can extend the function $i=(i_1,i_2)$ on the curve $\gamma_2$
keeping its one-to-one property. In the following, we denote by
$\tilde{i}(\cdot,\cdot)$ the extension of $i(\cdot,\cdot)$ to the domain
${D(i_1,i_2)}\cup \gamma_2$. In fact we can extend the function $i=(i_1,i_2)$,
preserving  the definition \eqref{lii},   in all $\mathbb{R}_+^2$, which is the
convex hull of ${D(i_1,i_2)}$, but it does not remain a one-to-one function and
also we do not have a mechanical interpretation for this choice. However, we
intend to construct an energy function
$\widehat{\psi}:\mathbb{R}_+^2\rightarrow\mathbb{R}_+^2$ which is convex in all
$\mathbb{R}_+^2$, the  convex hull of ${D(i_1,i_2)}$, using the above results.

First, we extend the function $\psi$ to ${D(i_1,i_2)}\cup \gamma_2$ by
\begin{align}
\widetilde{\psi}(i_1,i_2)=\left\{
\begin{array}{ll}
\psi(i_1,i_2)& \ \ \text{if} \ \ \ (i_1,i_2)\in {D(i_1,i_2)},\vspace{2mm}\\
1& \  \ \text{if} \ \ \ (i_1,i_2)\in \gamma_2.
\end{array}
\right.
\end{align}
The function $\widetilde{\psi}$ preserves the continuity property of $\psi$. One can see this fact more clearly, by using  that
\begin{align}
i^{-1}(i_1,i_2)=\left(\frac{i_1+\sqrt{i_1^2-4i_2}}{2},\frac{i_1-\sqrt{i_1^2-4i_2}}{2}\right)\in {D(\lambda_1,\lambda_2)}\cup \gamma_1, \quad \text{for all} \quad (i_1,i_2)\in {D(i_1,i_2)}\cup \gamma_2
\end{align}
and the definition of $g(\cdot,\cdot)$. Hence, we have
\begin{align}
\widetilde{\psi}(i_1,i_2)= e^{\frac{k}{2}\log^2 \frac{i_1+\sqrt{i_1^2-4i_2}}{i_1-\sqrt{i_1^2-4i_2}}}& \ \ \text{for all} \ \ \ (i_1,i_2)\in {D(i_1,i_2)}\cup\gamma_2\,
\end{align}
and the continuity of $\widetilde{\psi}(\cdot,\cdot)$ follows.

The function $\widetilde{\psi}(\cdot,\cdot)$ satisfies the condition \eqref{condBES} from  Theorem \ref{ThBES} if and only if  $k\geq \dd\frac{1}{3}$ because for these values of $k$ it  is convex in every convex open domain $\omega\subset {D(i_1,i_2)}\cup\gamma_2$ and it is a continuous function. It is bounded linearly from below by $0$.
On the other hand, from the definition of
$\widetilde\psi$, we have
\begin{align}
\min\left\{\sum\limits_{i=1}^r\mu_i \widetilde{\psi}(x_i)\ : \ x=\sum\limits_{i=1}^r\mu_i x_i,\
\sum\limits_{i=1}^r\mu_i=1,\ 1\leq r<\infty\right\}=1.
\end{align}
\newpage

Hence, we  conclude:
\begin{proposition}
The elastic energy $\widehat{\psi}:[0,\infty)\times \mathbb{R}_+\rightarrow\mathbb{R}_+$ defined by
\begin{align}\label{defpsifin}
\widehat{\psi}(i_1,i_2)=\left\{
\begin{array}{ll}
\dd e^{\frac{k}{2}\log^2 \frac{i_1+\sqrt{i_1^2-4i_2}}{i_1-\sqrt{i_1^2-4i_2}}}& \ \ \text{if} \ \ \ (i_1,i_2)\in {D(i_1,i_2)}\cup\gamma_2,\vspace{2mm}\\
1& \  \ \text{if} \ \ \ (i_1,i_2)\in([0,\infty)\times \mathbb{R}_+)\setminus ({D(i_1,i_2)}\cup\gamma_2)
\end{array}
\right.
\end{align}
is convex if and only if  $k\geq \dd\frac{1}{3}$.
\end{proposition}

Using the  sum of squared logarithms
inequality given by Theorem \ref{mon-dev}, we deduce:
\begin{proposition}\label{mon-dev}{\rm (The exponentiated sum of squared logarithms
inequality and monotonicity)}\\
The function $F\mapsto e^{k\,\norm{\dev_n \log U}^2},\ k\geq 0$ is separately monotone in $i_1, i_2$ for $n=3$ and monotone in $i_1$ for $n=2$.
\end{proposition}
\newcommand{\ihat}{\widehat{i}}
\begin{proof}
 In this proof, we will restrict ourselves to the case $n=3$ and show slightly more than seperate monotonicity.
 To this aim, let $i_1=\lambda_1+\lambda_2+\lambda_3$, $i_2=\lambda_1\lambda_2+\lambda_1\lambda_3+\lambda_2\lambda_3$, $i_3=\lambda_1\lambda_2\lambda_3$ and $\lambdahat_1, \lambdahat_2, \lambdahat_3$ analogously corresponding to $\ihat_1$, $\ihat_2$, $\ihat_3$ be given in such a way that $i_1\leq \ihat_1$, $i_2\leq \ihat_2$ and $i_3=\ihat_3$. Then these inequalities coincide with those from \eqref{sumofsquaredlog_condition} and Theorem \ref{Noging} and monotonicity of the exponential function yield $e^{k(\log^2\lam_1+\log^2\lam_2+\log^2\lam_3)}\leq e^{k(\log^2\lambdahat_1+\log^2\lambdahat_2+\log^2\lambdahat_3)}$.
The proof of the proposition follows from the equality
\begin{align*}
 e^{k\,\norm{\dev_n \log U}^2}&=e^{k\,\norm{ \log U}^2-\frac{k}{n}\,\tr[\log
U]^2}=e^{k\,\norm{\log U}^2}\cdot e^{-\frac{k}{n}
\log^2(\det U)},
\end{align*}
where we have shown the monotonicity of the first factor by the sum of squared
logarithms inequality and the second factor is independent of $i_1, i_2$. (And independent of $i_1$ in the analogous proof for $n=2$.)
\end{proof}

This holds for dimensions $n=2$ and $n=3$ and indeed in any dimension, in which the sum of squared logarithms inequality holds, see Conjecture \ref{conjecture_sumofsquaredlog}.
Therefore $\widehat{\psi}$ satisfies the criterion  of Steigmann from Lemma \ref{sc2} and in consequence we have our main result:
\begin{proposition}
The map $W:{\rm GL}^+(2)\rightarrow\mathbb{R}_+$ defined by
$
W(F)=e^{k\,\|{\rm dev}_2\log U\|^2}
$
is polyconvex for  $k\geq\dd\frac{1}{3}$.
\end{proposition}

\subsection{Polyconvexity of the volumetric response $F\mapsto e^{\widehat{k}(\log\det F)^m}$ in arbitrary dimensions}\label{polytr}
 In a previous work \cite[Proposition 5.11]{NeffGhibaLankeit} we have established the following lemma:
\begin{lemma}
The function
$
 t\mapsto e^{\widehat{k}\,(\log t)^m}
$
is convex if and only if
$
 \widehat{k}\geq \tel{m^{(m+1)}}.
$
\end{lemma}
This implies:
\begin{proposition}\label{propvol} For $m\in\N$ the function
$
 F\mapsto e^{\widehat{k}\,(\log\det F)^m}, \; F\in {\rm GL}^+(n),
$
is polyconvex for $\widehat{k}\geq \tel{m^{(m+1)}}$.
Explicitly evaluating this condition in the case of $m=2$, we arrive at
$\widehat{k}\geq\tel8$.
\end{proposition}
\subsection{The main  polyconvexity statement}

In view of the results established in Subsection \ref{poly-isochoric} and \ref{polytr} we conclude that:

\begin{theorem}\label{mainth} The functions $W_{_{\rm eH}}:\R^{n\times n}\to \R$ from the family of exponentiated Hencky type energies
\begin{align}
 W_{_{\rm eH}}(F)&=W_{\text{\rm iso}}(\frac F{\det F^{\frac{1}{n}}})+W_{\text{\rm vol}}(\det F^{\tel n}\id) = \left\{\begin{array}{lll}
\dd\frac{\mu}{k}\,e^{k\,\|{\rm dev}_n\log U\|^2}+\frac{\kappa}{2\widehat{k}}\,e^{\widehat{k}\,[(\log \det U)]^2}&\text{if}& \det\, F>0,\vspace{2mm}\\
+\infty &\text{if} &\det F\leq 0,
\end{array}\right.
\end{align}
are polyconvex for $n=2$, $\mu>0, \kappa>0$, $k\geq\dd\frac{1}{3}$ and $\widehat{k}\dd\geq \tel8$.
\end{theorem}

\section{Unconditional coercivity: coercivity for every exponent $1\leq q<\infty$}
\setcounter{equation}{0}

We start the analysis of  coercivity problem by considering the simple one dimensional case:
\begin{align}
e^{(\log t)^2}=(e^{\log t})^{\log t}=t^{\log t}.
\end{align}
Since for some particular choices (see Figure \ref{starting1})
\begin{figure}[h!]
\centering
\includegraphics[scale=0.7]{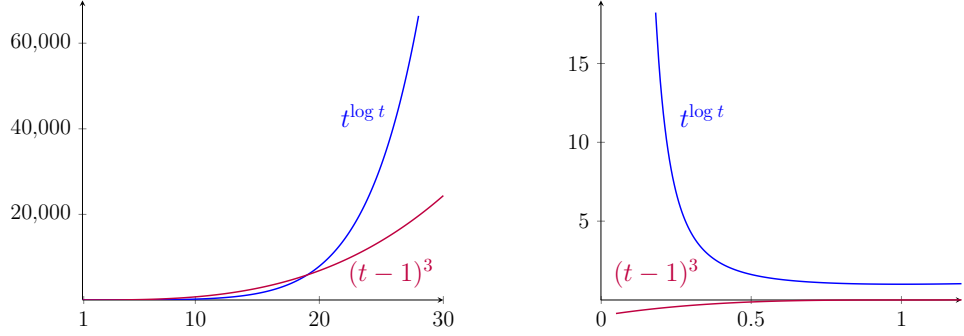}
\centering
\caption{\footnotesize{ The function $e^{(\log t)^2}$ dominates $|t-1|^3$ on $(1,\infty)$ as well as on $(0,1)$.}}
\label{starting1}
\end{figure}%
we see that for large values of $t$, the function $e^{(\log t)^2}$ dominates $|t-1|^\alpha$,
for arbitrary $\alpha>0$. However, $(\log t)^2$ alone does not satisfy any growth condition of this type.  This motivates:
\begin{lemma}{\rm (Unconditional coercivity)}
 For all $\alpha>0$ there exists $K>0$ such that for all $t>0$
 \begin{equation}\label{unucoer}
   e^{(\log t)^2}\geq K\,|t-1|^\alpha.
 \end{equation}
\end{lemma}
\begin{proof}
 First we consider ``large'' values of $t$, use the substitution $s=\log t$ and observe that
\[
 \widehat K=\inf_{t>1}\frac{e^{\log^2 t}}{(t-1)^\alpha}=\inf_{s>0}\frac{e^{s^2}}{(e^s-1)^\alpha}\geq \inf_{s>0} \frac{e^{s^2}}{e^{\alpha s}}
 =\inf_{s>0} e^{s^2-\alpha s}
\]
is positive,
because $\inf\limits_{s>0}( s^2-\alpha s)>-\infty$. Secondly,
$
 \inf_{t\in(0,1)}\frac{e^{\log^2t}}{(1-t)^\alpha}
 = 1>0.
$
Hence the claim follows upon the choice $K=\min\set{\widehat K,1}$.
\end{proof}

\begin{cor}
\label{thm:coerccor}
For all $\alpha, \beta>0$ there is $K>0$ such that for all $t>0$
  \[
e^{\beta(\log t)^2}\geq K\,|t-1|^{\alpha\beta}.
  \]
\end{cor}
\begin{lemma}
 \label{thm:lemmataylor}
 Let $a>0$ and $\gamma>0$. Then there exist  positive constants $C,\Ktilde$ such that for all $t\in(0,a)$ and for all $s>1$:
\[
 C\sqrt{s+t}^{\,\gamma} -\Ktilde\leq \sqrt{s}^{\,\gamma}.
\]
\end{lemma}
\begin{proof}
 Choose $m\nat_0$ such that $\gamma-2m>0\geq\gamma-2m-2$. Let $t\in(0,a)$ and $s>1$.
 Taylor's expansion shows the existence of some $\xi\in(0,t)\subset(0,a)$ such that
\begin{align*}
 \sqrt{s+t}^{\,\gamma}&=\sqrt{s}^{\,\gamma}+{\,\frac\gamma2}
 \sqrt{s}^{{\,\gamma}-2}t+\frac{{\,\gamma}({\,\gamma}-2)}{2^2\cdot2!}\sqrt{s}^{{
\, \gamma }
-4}t^2+\ldots+\frac{{\,\gamma}({\,\gamma}-2)\cdots({\,\gamma}-2m+2)}{2^m m!}
\sqrt { s } ^ {{\,\gamma}-2m}t^m\\
 &\qquad
+\frac{{\,\gamma}({\,\gamma}-2)\cdots({\,\gamma}-2m)}{2^{m+1}(m+1)!}\sqrt{s+\xi}
^ { { \, \gamma}-2m-2}t^{m+1}\\
 &\leq\left[1+t{\,\frac\gamma2}+t^2\frac{{\,\gamma}({\,\gamma}-2)}{2^2\cdot 2!}
+\ldots+t^m\,\frac{{\,\gamma}({\,\gamma}-2)\cdots({\,\gamma}-2m+2)}{2^m
m!}\right]\sqrt{s}^{\,\gamma}\\
 &\qquad+\frac{{\,\gamma}({\,\gamma}-2)\cdots({\,\gamma}-2m)}{2^{m+1}(m+1)!}\,t^{
m+1 }\,
\underbrace{\sqrt{s}^{{\,\gamma}-2m-2}}_{\leq 1}.
\end{align*}
and the lemma follows upon the choices of
\begin{align*}
 \frac1C:&=\left[1+a{\,\frac\gamma2}+a^2\frac{{\,\gamma}({\,\gamma}-2)}{2^2\cdot
2!}
+\ldots+a^m\frac{{\,\gamma}({\,\gamma}-2)\cdots({\,\gamma}-2m+2)}{2^m
m!}\right]\,,\
 \frac{\Ktilde}{C}:=\frac{{\,\gamma}({\,\gamma}-2)\cdots({\,\gamma}-2m)}{2^{m+1
} (m+1)!}\,a^{m+1 }.
  \qedhere
\end{align*}
\end{proof}

Moreover, from Lemma \ref{thm:lemmacoercive} we see that there cannot be
any polynomial upper bound $C(1+\|F\|^q)\geq W(F)$.
\begin{lemma}
\label{thm:lemmacoercive}
Let $\alpha, \beta>0$. Then there are constants $K_1, K_2>0$ such that for all
$\lam_1$, $\lam_2\in\Rplus$
\[ e^{\beta(\log^2\lam_1+\log^2\lam_2)}\geq
K_1\Big((\lam_1-1)^2+(\lam_2-1)^2\Big)^\frac{\alpha\beta}{2}-K_2\]
\end{lemma}

\begin{proof}
 For $\lam_1,\lam_2\in(0,3]$,
 $
  e^{\beta(\log^2\lam_1+\log^2\lam_2)}\geq 0
$
 and the claim follows, even for arbitrary large $K_1$, by setting
 \[
  K_2:=\sup_{(\lam_1,\lam_2)\in[0,3]\times[0,3]}\left\{ K_1((\lam_1-1)^2+(\lam_2-1)^2)^{\frac{\alpha\beta}2}\right\},
 \]
 which is finite by continuity of $\lam\mapsto K_1((\lam_1-1)^2+(\lam_2-1)^2)^{\frac{\alpha\beta}2}$ and compactness of $[0,3]\times[0,3]$.

 For $\lam_1,\lam_2\in[3,\infty)$ note that
\begin{align*}
 (\lam_1-1)^2(\lam_2-1)^2&=(\lam_1-1)^2\frac{(\lam_2-1)^2}2+\frac{(\lam_1-1)^2}2(\lam_2-1)^2\geq (\lam_1-1)^2+(\lam_2-1)^2
\end{align*}
and hence
\begin{equation}
 \label{eq:coercprodukt}
 [(\lam_1-1)(\lam_2-1)]^{\alpha\beta}\geq ((\lam_1-1)^2+(\lam_2-1)^2)^{\frac{\alpha\beta}2}.
\end{equation}
Using Corollary \ref{thm:coerccor}, we obtain $K>0$ fulfilling
\begin{align*}
 e^{\beta(\log^2\lam_1+\log^2\lam_2)}&=e^{\beta\log^2\lam_1}e^{\beta\log^2\lam_2}\geq K^2(\lam_1-1)^{\al\beta}(\lam_2-1)^{\al\beta}\overset{\eqref{eq:coercprodukt}}{\geq} K^2 ((\lam_1-1)^2+(\lam_2-1)^2)^{\frac{\al\beta}2}.
\end{align*}
Now let us consider the last possible case:  $\lam_1\geq 3,\ \lam_2\in(0,3)$.
Then Corollary \ref{thm:coerccor} and Lemma \ref{thm:lemmataylor} with
$s=(\lam_1-1)^2$, $\gamma=\alpha\beta$, $t=(\lam_2-1)^2$ and $a=4$ yield $K,C,\Ktilde>0$ such that
\begin{equation*}
 e^{\beta\log^2\lam_1+\beta\log^2\lam_2}\geq K(\lam_1-1)^{\al\beta} \underbrace{ e^{\beta \log^2\lam_2} }_{\geq 1}
\geq KC((\lam_1-1)^2+(\lam_2-1)^2)^{\frac{\al\beta}{2}} - K\Ktilde.
\end{equation*}
 Finally, we choose the smallest $K_1$ and largest $K_2$ required by any of these individual cases.
\end{proof}

\begin{remark}
 \label{thm:remcoerchigherdim}
 The same can be done in dimension $n=3$ or higher. For larger $n$, the domain must be split into the single cases in a different way, replacing $3$ as the separating number (for $n=2$, indeed $\sqrt{2}+1$ would have sufficed) and Lemma \ref{thm:lemmataylor} must be applied several times.
\end{remark}

\begin{theorem}
 Regardless of dimension $n\nat$ and $\beta>0$,
$
  e^{\beta\norm{\log U}^2}
$
 is {\bf unconditionally coercive}, in the sense that for arbitrary $\al>0$ there are constants $K_1,K_2>0$ such that
 \begin{equation}
  \label{eq:coerc}
  e^{\beta\norm{\log U}^2}\geq K_1\norm{U-\id}^{\alpha\beta}-K_2.
 \end{equation}
\end{theorem}
\begin{proof}
 Unitarily diagonalizing the symmetric positive definite matrix $U$ equivalently transforms equation \eqref{eq:coerc} into
 \[
  e^{\beta\sum_{i=1}^n\log^2\lam_i}\geq K_1\left(\sum_{i=1}^n
(\lam_i-1)^2\right)^{\frac{\alpha\beta}2}-K_2.
 \]
 Here we can apply Lemma \ref{thm:lemmacoercive} (or Remark \ref{thm:remcoerchigherdim} for $n>2$).
\end{proof}
\begin{remark} Let $n=2, k>0$ and consider
 $
  e^{k\,\norm{\dev_2 \log U}^2}.
$
  This energy is not coercive in the following sense: neither are there constants $K_1,K_2,\alpha >0$ such that
 \begin{equation}
  \label{eq:devuncoercive1}
  e^{k\,\norm{\dev_2 \log U}^2}\geq K_1\norm{U-\id}^{\alpha k}-K_2,
 \end{equation}
 nor do there exist $K_1,K_2, \alpha>0$ such that
 \begin{equation}
  \label{eq:devuncoercive2}
  e^{k\norm{\dev_2 \log U}^2}\geq K_1\norm{\dev_2 U}^{\alpha k}-K_2.
 \end{equation}
\end{remark}
\begin{proof}
 Suppose there were $K_1, K_2, \alpha$ satisfying \eqref{eq:devuncoercive1}, i.e. for all $\lam_1,\lam_2>0$:
 \[
  e^{\frac k2\,\log^2\frac{\lam_1}{\lam_2}}\geq K_1((\lam_1-1)^2+(\lam_2-1)^2)^{\frac {\alpha k}{2}}-K_2.
 \]
 Choose $\lam_1=\lam_2=N+1$. This would lead to
$
  1\geq K_1(2N^2)^{\frac {\alpha k}{2}}-K_2\to \infty.
$
 In the same manner, \eqref{eq:devuncoercive2} corresponds to
 \[
  e^{\frac k2\,\log^2\frac{\lam_1}{\lam_2}}\geq \frac{K_1}{2^\frac {\alpha k}{2}} |\lam_1-\lam_2|^{\alpha k}-K_2\, .
 \]
 Choose $\lam_2=\frac{\lam_1}2=N$ to obtain a contradiction by
$
  e^{\frac k2\log^22}\geq \frac{K_1}{2^\frac {\alpha k}{2}} N^{\alpha k}-K_2\to \infty.\qedhere
$
\end{proof}

However, we have the following results which will finally lead to the coercivity of $W(U)$:
\begin{lemma}\label{codevtr}
 Assume $\mu>0,\ \kappa>0$. For arbitrary dimension $n\nat$ and $k,\widehat{k}>0$,
 and for arbitrary $\al_1,\al_2>0$ there are constants $C_1,C_2,C_3>0$ such that for any $U\in PSym(n)$
 \begin{equation}
  \label{eq:coerc-dev}
 \widehat{W}_{_{\rm eH}}(U)= \frac{\mu}{k}\,e^{k\,\|\dev_n\log U\|^2}+\frac{\kappa}{2\widehat{k}}\,e^{\widehat{k}\,|{\rm tr}(\log U)|^2}\geq C_1\left\|\frac{U}{\det U^{1/n}}-\id\right\|^{\alpha_1k}+C_2|\det U-1|^{\alpha_2\widehat{k}}-C_3.
 \end{equation}
\end{lemma}
\begin{proof}
Let us repeat that
\begin{align}
\dev_n\log U=\log U-\frac{1}{n}\,\tr(\log U)\cdot \id=\log U-\frac{1}{n}\,\log(\det U)\cdot \id=\log \frac{U}{\det U^{1/n}}.
\end{align}
Hence, using \eqref{eq:coerc} we know that for arbitrary $\al_1>0$ there are constants $K_1,K_3>0$ such that
\begin{align}
e^{k\,\|\dev_n\log U\|^2}=e^{k\,|\log \frac{U}{\det U^{1/n}}|^2}\geq K_1\,\left\| \frac{U}{\det U^{1/n}}-\id\right\|^{\alpha_1k}-K_3.
\end{align}
On the other hand, using Corollary \ref{thm:coerccor} we obtain that for arbitrary $\alpha_2>0$ there is the constant $K_2>0$ such that
\begin{equation}\label{unucoer2}
 e^{\widehat{k}|\tr(\log U)|^2}\geq K_2\,|\det U-1|^{\alpha_2\widehat{k}}.
 \end{equation}
 With the choices $C_1= \frac{\mu}{k}\,K_1, C_2=\frac{\kappa}{2\widehat{k}} \,K_2, C_3=\frac{\mu}{k}\,K_3$,  the proof is complete.
\end{proof}

Using a technique similar to that used in   \cite{Hartmann_Neff02}
we obtain:

\begin{theorem}\label{thuncoe} Assume $\mu>0,\ \kappa>0$. Regardless of dimension $n\nat$ and $k,\widehat{k}>0$,
 and for arbitrary $q\geq 1$ there are  the constants $K_1,\,K_2>0$ such that for all $U\in PSym(n)$
\begin{align}\label{coetott}
  \widehat{W}_{_{\rm eH}}(U)=\frac{\mu}{k}\,e^{k\,\|\dev_n\log U\|^2}+\frac{\kappa}{2\widehat{k}}\,e^{\widehat{k}\,|{\rm tr}(\log U)|^2}\geq K_1\|U-\id\|^q-K_2.
\end{align}
\end{theorem}
\begin{proof}
Using the inequality
$|a+b|^q\leq 2^{q-1}\,(|a|^q+|b|^q)$ for all $a,b>0$, and $q\geq1$,
we deduce
\begin{align}
\|U-\id\|\,^q&=\left\|\left(\frac{U}{\det{U}^{1/n}}-\id\right)\det{U}^{1/n}+
\det{U}^{1/n}\cdot \id-\id\right\|^q\notag\\
&\leq \left[\left\|\frac{U}{\det{U}^{1/n}}-\id\right\|\,|\det{U}|^{1/n}+n
|\det{U}^{1/n}-1|\,\right]^q\\
&\leq 2^{q-1}\left[\left\|\frac{U}{\det{U}^{1/n}}-\id\right\|^q\,|\det{U}|^{q/n}+n^q
|\det{U}^{1/n}-1|^q\right]\notag
.\end{align}
Young's inequality 
leads  to \begin{align}
\|U-\id\|\,^q&\leq 2^{q-1}\left[\left\|\frac{U}{\det{U}^{1/n}}-\id\right\|^q\, |\det{U}|^{q/n}+n^q |\det{U}^{1/n}-1|^q\right]\notag\\
&\leq 2^{q-1}\left[\frac{1}{2}\left\|\frac{U}{\det{U}^{1/n}}-\id\right\|^{2q}+\frac{1}{2}|\det{U}|^{2q/n}+ \frac{n^q}{2} |\det{U}^{1/n}-1|^{2q}+\frac{n^q}{2}\right],
\end{align}
which entails

\begin{align}
\|U-\id\|\,^q&\leq 2^{q-1}\left[\frac{1}{2}\|\frac{U}{\det{U}^{1/n}}-\id\|^{2q}+\frac{1}{2}(|\det{U}^{1/n}-1|+1)^{2q}+ \frac{n^q}{2} |\det{U}^{1/n}-1|^{2q}+\frac{n^q}{2}\right]\notag\\
&\leq 2^{q-2}\left[\|\frac{U}{\det{U}^{1/n}}-\id\|^{2q}
+2^{{2q}-1}(|\det{U}^{1/n}-1|^{2q}+1)+ n^q |\det{U}^{1/n}-1|^{2q}+n^q\right]\notag\\
&= 2^{q-2}\left[\|\frac{U}{\det{U}^{1/n}}-\id\|^{2q}
+\left(n^q
+2^{{2q}-1}\right)|\det{U}^{1/n}-1|^{2q}
+n^q+2^{{2q}-1}\right].\label{ineq415}
\end{align}

\newcommand{\khat}{\widehat{k}}
 Let $C_1, C_2, C_3 >0$ be as provided upon an application of Lemma \ref{codevtr} with the choices of $\al_1=2q/k$, $\al_2=2q/\khat$, and define $A_1=\max\{\frac{2^{q-2}}{C_1}, \frac{2^{q-2}n^q+2^{3q-3}}{C_2}\}$ and $A_2=2^{q-2}n^q+2^{3q-3}$. Then \eqref{ineq415} leads to
\begin{align}
\|U-\id\|\,^q
&\leq A_1\left[C_1\|\frac{U}{\det{U}^{1/n}}-\id\|^{\alpha_1 k}
+C_2|\det{U}^{1/n}-1|^{\alpha_2\widehat{k}}\right]+A_2,
\end{align}
thus by definition of $C_1, C_2, C_3$, the inequality given by Lemma \ref{codevtr} can be used to deduce
\begin{align}
\|U-\id\|\,^q
&\leq A_1\left[ \frac{\mu}{k}\,e^{k\,\|\dev_n\log U\|^2}+\frac{\kappa}{2\widehat{k}}\,e^{\widehat{k}\,|{\rm tr}(\log U)|^2}+C_3\right]+A_2,
\end{align}
and further
\begin{align}
A^{-1}_1\|U-\id\|^q-C_3-A_1^{-1}A_2
&\leq  \frac{\mu}{k}\,e^{k\,\|\dev_n\log U\|^2}+\frac{\kappa}{2\widehat{k}}\,e^{\widehat{k}|{\rm tr}(\log U)|^2}.
\end{align}
Choosing  $K_1=A^{-1}_1$ and $K_2=C_3+A_1^{-1}A_2$, we obtain the inequality
\eqref{coetott}, and the proof is complete.
\end{proof}

\begin{Def} {\rm (Coercivity)} Let $I(\varphi)$ be the elastic stored energy
functional depending on the deformation $\varphi(x,t)$. We say that $I$ is $q$-coercive (for $q\geq 1$)
whenever {for all $K>0$ there is some $\widetilde{K}>0$ such that }for any $\varphi\in W^{1,q}(\Omega,\R^n)$
\begin{align}
I(\varphi)\leq K \,{\Rightarrow} \, \|\nabla \varphi\|_{L^q(\Omega)}\leq
\widetilde{K}.
\end{align}
\end{Def}
A direct consequence of Theorem \ref{thuncoe} is the following result:
\begin{theorem}  Assume for the elastic moduli $\mu>0,\ \kappa>0$ and $k>0, \ \widehat{k}>0$.  Consider  the energy
\begin{align}
I(\varphi)=\int_\Omega W_{_{\rm eH}}(\nabla \varphi(x)) \,dx
\end{align}
 where
$
W_{_{\rm eH}}(F)=\widehat{W}_{_{\rm eH}}(U)= \frac{\mu}{k}\,e^{k\,\|\dev_2\log U\|^2}+\frac{\kappa}{2\,\widehat{k}}\,e^{\widehat{k}\,|{\rm tr}(\log U)|^2}.
$ Then $I(\varphi)$ is $q$-coercive for any $1\leq q<\infty$.
\end{theorem}

\section{The static problem in the planar case}
\subsection{Formulation of the static problem in the planar case}\label{Formulation}\setcounter{equation}{0}
The static problem in the planar case consists in finding the solution $\varphi$ of the equilibrium  equation
\begin{align}\label{exst}
0={\rm Div} \,S_1(\nabla \varphi)\qquad  \text{in} \qquad \Omega\subset
{\R^2},
\end{align}
where the first Piola-Kirchhoff stress tensor corresponding to the energy $W_{_{\rm eH}}(F)$ is given by the constitutive equation
\begin{align}\label{pks}
&S_1(F)=\left[2{\mu}\,e^{k\,\|\dev_2\log\,U\|^2}\cdot \dev_2\log\,U+{\kappa}\,e^{\widehat{k}\,[\tr(\log U)]^2}\,\tr(\log U)\cdot \id\right]F^{-T},  \qquad x\in\overline{\Omega},
\end{align}
with $F=\nabla \varphi, \ U= \sqrt{F^TF}$. The above system of equations is supplemented, in the
case of the mixed problem,  by the boundary conditions
\begin{align}\label{cl1}
{\varphi}({x})&=\widehat{\varphi}_i({x}) \qquad \text{ on  }\quad \Gamma_D,\qquad
{S}_1({x}).\,n=\widehat{s}_1({x}) \qquad \text{ on  }\quad \Gamma_N,\notag
\end{align}
where $\Gamma_D,\Gamma_N$  are subsets of the boundary $\partial \Omega$, so that $\Gamma_D\cup\overline{\Gamma}_N=\partial \Omega$,
$\Gamma_D\cap{\Gamma}_N=\emptyset$, ${n}$ is the unit outward normal to the boundary and  $\widehat{\varphi}_i, \widehat{s}_1$ are prescribed fields.

\subsection{Existence of minimizers in plane elastostatics}

In plane elastostatics, having proved the coercivity and polyconvexity of the energy $W(U)$, it is a standard matter to prove the existence
of a minimizer.

\begin{theorem}\label{mainexist}{\rm (Existence of minimizers)} Let the reference configuration
$\Omega\subset \mathbb{R}^2$ be a bounded smooth domain and let $\Gamma_D$ be a non-empty and relatively open part of the boundary  $\partial \Omega$. Assume that
$I(\varphi)=\int_\Omega W_{eH}(\nabla \varphi(x)) dx$ where
$
W_{_{\rm eH}}(F)=\widehat{W}_{_{\rm eH}}(U)= \frac{\mu}{k}\,e^{k\,\|\dev_2\log U\|^2}+\frac{\kappa}{2\widehat{k}}\,e^{\widehat{k}\,|{\rm tr}(\log U)|^2}.
$
Let $\varphi_0\in W^{1,q}(\Omega), \ q\geq1$ be given with $I(\varphi_0)<\infty$ and  $\mu>0,\ \kappa>0$, $k>\frac{1}{3}$ and $\widehat{k}>\frac{1}{8}$. Then the problem
\begin{align}
\min\left\{I(\varphi)=\int_\Omega W_{_{\rm eH}}(\nabla \varphi(x)) dx, \ \varphi=\varphi_0 \quad \text{on} \ \Gamma_D\subset \partial \Omega,\quad  \varphi\in W^{1,q}(\Omega)\right\}
\end{align}
admits at least one solution $\varphi$. Moreover, $\varphi\in W^{1,p}(\Omega)$ for all $p\geq1$.
\end{theorem}
\begin{remark}
 Formally, this solution corresponds to a solution of the boundary-value problem formulated in Subsection \ref{Formulation}. However, the minimizing property of $\varphi$ alone is not sufficient to show that the Euler-Lagrange equation \eqref{exst} is satisfied by $\varphi$ in a weak sense: since we do not know whether $\det \nabla \varphi\geq c>0$, it is not clear whether the energy functional is Frech\'{e}t-differentiable at the minimizer.
\end{remark}

\begin{remark}
While the parameters $\mu,\, \kappa>0$ are already uniquely determined from the infinitesimal material response,
$k,\widehat{k}>0$ can be used to fit some nonlinear aspects of the response. This will be done in a future contribution.
\end{remark}

\section{The three-dimensional case: $F\mapsto e^{k\,\|\dev_3\log U\|^2}$}\setcounter{equation}{0}

The 3D-case is, as usual, much more involved.
As was previously shown \cite{NeffGhibaLankeit}, the exponentiated Hencky energy
\[
	F\mapsto \frac{\mu}{k}\,e^{k\, \|\dev_n\log U\|^2}\,, \quad k>\frac{3}{16},\; \mu>0
\]
in dimension $n=3$ is not rank-one convex and therefore not polyconvex. However, numerical results strongly suggest that $W_{\rm eH}$ is, in fact, rank-one convex on a cone of the form
\[
	\mathcal{E} =\{U\in{\rm PSym}(3)\,\big|\, \|\dev_3\log U\|^2<\frac{2}{3}\, \widetilde{\boldsymbol{\sigma}}_{\!\mathbf{y}}^2\,\}.
\]
with $\widetilde{\boldsymbol{\sigma}}_{\!\mathbf{y}} \gg 1$. This convexity property is of particular interest in the theory of plasticity, since the loss of rank-one convexity occurs only for strains which induce permanent deformations. We will discuss the possible application of the exponentiated Hencky energy in plasticity theory in the near future \cite{NeffGhibaPlasticity}.

\section{Summary and open problems}
\setcounter{equation}{0}
To summarize, in the present paper
\begin{itemize}
\item We have applied Steigmann's polyconvexity condition and proved that the planar exponentiated Hencky-strain energy function
\begin{align}
F\mapsto W_{_{\rm eH}}(F):=\widehat{W}_{_{\rm eH}}(U):&=\left\{\begin{array}{lll}
\frac{\mu}{k}\,e^{k\,\|\dev_2\log\,U\|^2}+\frac{\kappa}{2\widehat{k}}\,e^{\widehat{k}\,(\tr(\log U)^2}&\ \ \ \!\!\text{if}& \det\, F>0,\vspace{2mm}\\
+\infty &\ \ \ \!\! \text{if} &\det F\leq 0\end{array}\right.
\end{align}
is {\bf polyconvex } for  $\mu>0, \kappa>0$, $k\geq\dd\frac{1}{3}$ and $\widehat{k}\dd\geq \tel8$.
\item We have shown that the exponentiated volumetric energy function
\begin{align}
 F\mapsto \frac{\kappa}{2\widehat{k}}\,e^{\widehat{k}\,(\tr(\log U))^m}, \quad
F\in {\rm GL}^+(n)
\end{align}
is {\bf polyconvex } w.r.t $F$ for $\widehat{k}\geq \tel{m^{(m+1)}}$.
\item  We have proven that, regardless of dimension $n\nat$ and $k,\widehat{k}>0$, the energies of the family
$F\mapsto W_{_{\rm eH}}(F)$  satisfy  $q$-growth coercivity for any $1\leq q<\infty$, and therefore allow  in the planar case $n=2$ for a
complete existence theory based on the direct methods of the calculus of variations.
\end{itemize}

Using the terminology from  \cite{Neff_Osterbrink_Martin_hencky13,,Neff_Nagatsukasa_logpolar13}, in {the present} paper we have shown polyconvexity  of
\begin{align}
W_{_{\rm eH}}(F):=\frac{\mu}{k}\,e^{k\,{\rm dist}^2_{{\rm geod,{\rm SL}(2)}}\left( \frac{F}{\det F^{1/2}},\, {\rm SO}(2) \right)}+\frac{\kappa}{2\widehat{k}}\,e^{\widehat{k}\,{\rm dist}^2_{{\rm geod},\mathbb{R}_+\cdot \id}\left(\det F^{1/2}\cdot \id,\id\right)}.
\end{align}
and we have proved the existence of the solution of the corresponding minimization problem.

\medskip

In the first part \cite{NeffGhibaLankeit} of this paper we have shown rank-one convexity for $k\geq \frac{1}{4}$. Here, we have obtained polyconvexity for $k\geq \frac{1}{3}$. Hence, a first open problem is to investigate the gap $\frac{1}{3}> k\geq \frac{1}{4}$.

\medskip


 Results obtained by Pipkin \cite{Pipkin}, concerning convexity conditions
when $F$ is a $3\times 2$ matrix, may be used to extend our polyconvexity
results to membrane theory. The associated stretch tensor is $U=\sqrt{F^{T}F}%
,$ which is still a $2\times 2$ matrix, just as in the case of plane strain
considered here. The results of \cite{Pipkin} ensure that polyconvexity with respect
to $2\times 2$ deformation gradients - established here for the family $W_{_{\rm eH}}$
- yield polyconvexity of the same energy with respect to the $3\times 2$
deformation gradients of membrane theory \cite{steigmann1990tension}, provided that the first Piola-Kirchhoff stress $S_1$ (the right-hand side of equation \eqref{pks}) is positive semi-definite. The latter restriction
is necessary for rank-one convexity (and hence also for polyconvexity) when $%
F$ is a $3\times 2$ matrix. However, this is not enough to yield the
existence of minimizers, even in the presence of coercivity, because the
restriction on $F,$ required for a positive semi-definite stress, cannot be
guaranteed \textit{a priori.}

\section{Acknowledgement}
We  would like to thank Prof. Mircea B\^irsan  (University of
Duisburg-Essen) for indicating to us a technical point and showing  that $e^{\|\log U\|^2}$ does
not satisfy the Steigmann's sufficiency criterion in 3D and to Prof.
Bernard Dacorogna (EPFL-Lausanne) for sending us  reference
\cite{DacorognaMarechal} and  Prof.  Miroslav \v{S}ilhav\'{y} (Academy of Sciences of the Czech Republic, Prague) for
comments on rank-one convexity and polyconvexity.

\bibliographystyle{plain} 
\addcontentsline{toc}{section}{References}
\begin{footnotesize}

\end{footnotesize}

\appendix
\section*{Appendix}\addcontentsline{toc}{section}{Appendix} \addtocontents{toc}{\protect\setcounter{tocdepth}{-1}}

\setcounter{section}{1}

\subsection{About some convexity and polyconvexity conditions in 3D}\setcounter{equation}{0}
A first relation of convexity properties in the different representations of the energy $W$ is given by:

\begin{lemma}\label{lem-martin}
Let $\Psi:\R_+^3\to\R$ and $\Phi:\R_+^7\to\R$ with
\begin{align}\label{cond-lem-martin}
	\Phi(\lambda_1,\lambda_2,\lambda_3,\mu_1,\mu_2,\mu_3,\delta) = \Psi(\lambda_1+\lambda_2+\lambda_3,\:\mu_1+\mu_2+\mu_3,\:\delta)
\end{align}
for all $\lambda_1,\lambda_2,\lambda_3,\mu_1,\mu_2,\mu_3,\delta\in \R_+$. Then $\Phi$ is convex if and only if $\Psi$ is convex.
\end{lemma}

\begin{proof}
Assume that $\Psi$ is not convex. Then we can find $x,{\widehat{x}}\in \R_+^3$, $s\in[0,1]$ with
$
	\Psi(s\,x + (1-s)\,{\widehat{x}}) > s\,\Psi(x) + (1-s)\,\Psi({\widehat{x}})\,,
$
where $x,{\widehat{x}}$ have the form
 $
x=(a,b,\delta),\  {\widehat{x}}=({\widehat{a}},{\widehat{b}},{\widehat\delta})
$
with $a,b,\delta,{\widehat{a}},{\widehat{b}},\widehat{\delta}\in \R^+$. We choose $\lambda_1,\lambda_2,\lambda_3\in\R^+$ such that $\lambda_1+\lambda_2+\lambda_3=a$ (for example, choose $\lambda_1=\lambda_2=\lambda_3=\frac{a}{3}$) and, analogously, ${\widehat\lambda}_i$ as well as $\mu_i$  and $\widehat{\mu}_i$ with
$
	{\widehat\lambda}_1+{\widehat\lambda}_2+{\widehat\lambda}_3={\widehat{a}}\,,\  \mu_1+\mu_2+\mu_3=b\,,\ {\widehat\mu}_1+{\widehat\mu}_2+{\widehat\mu}_3={\widehat{b}}\,.
$

We define
$
	y=(\lambda_1,\lambda_2,\lambda_3,\mu_1,\mu_2,\mu_3,\delta), \ {\widehat{y}}=(\widehat{\lambda}_1,{\widehat\lambda}_2,{\widehat\lambda}_3,{\widehat\mu}_1,{\widehat\mu}_2,{\widehat\mu}_3,{\widehat\delta})
$
and find
\begin{align}
\Phi(s\,y+(1-s)\,{\widehat{y}})\notag
&= \Phi\Big(s\,\lambda_1+(1-s)\,{\widehat\lambda}_1,\: s\,\lambda_2+(1-s)\,{\widehat\lambda}_2,\: s\,\lambda_3+(1-s)\,{\widehat\lambda}_3,\label{eq:representationConvexityProofFails1}\\
	&\qquad\quad s\,\mu_1+(1-s)\,{\widehat\mu}_1,\: s\,\mu_2+(1-s)\,{\widehat\mu}_2,\: s\,\mu_3+(1-s)\,{\widehat\mu}_3, s\,\delta+(1-s)\,{\widehat\delta}\Big)\nonumber\\
	&= \Psi(s\,a+(1-s)\,{\widehat{a}},s\,b+(1-s)\,{\widehat{b}},s\,\delta+(1-s)\,{\widehat\delta})=
	 \Psi(s\,x + (1-s)\,{\widehat{x}})\\
&> s\,\Psi(x) + (1-s)\,\Psi({\widehat{x}})= s\,\Psi(a,b,\delta)+(1-s)\,\Psi({\widehat{a}},{\widehat{b}},{\delta})\nonumber\\
	&= s\,\Psi(\lambda_1+\lambda_2+\lambda_3,\: \mu_1+\mu_2+\mu_3,\: \delta) + (1-s)\, \Psi({\widehat\lambda}_1+{\widehat\lambda}_2+{\widehat\lambda}_3,\: {\widehat\mu}_1+{\widehat\mu}_2+{\widehat\mu}_3,\: {\widehat\delta})\nonumber\\
	&= s\,\Phi(\lambda_1,\lambda_2,\lambda_3,\mu_1,\mu_2,\mu_3,\delta) + (1-s)\,\Phi({\widehat\lambda}_1,{\widehat\lambda}_2,{\widehat\lambda}_3,{\widehat\mu}_1,{\widehat\mu}_2,{\widehat\mu}_3,{\widehat\delta})\nonumber\\
	&= s\,\Phi(y) + (1-s)\,\Phi({\widehat{y}})\,,\nonumber
\end{align}
and conclude that $\Phi$ is not convex.

\medskip

Now assume $\Psi$ to be convex. We use the same approach: for arbitrary
$
	y=(\lambda_1,\lambda_2,\lambda_3,\mu_1,\mu_2,\mu_3,\delta),$ \break $ {\widehat{y}}=({\widehat\lambda}_1,{\widehat\lambda}_2,{\widehat\lambda}_3,{\widehat\mu}_1,{\widehat\mu}_2,{\widehat\mu}_3,{\widehat\delta})
$
we find
\begin{align}
\Phi(s\,y+(1-s)\,{\widehat{y}})
	&= \Phi\Big(s\,\lambda_1+(1-s)\,{\widehat\lambda}_1,\: s\,\lambda_2+(1-s)\,{\widehat\lambda}_2),\: s\,\lambda_3+(1-s)\,{\widehat\lambda}_3,\nonumber\\
	&\qquad\quad s\,\mu_1+(1-s)\,{\widehat\mu}_1,\: s\,\mu_2+(1-s)\,{\widehat\mu}_2),\: s\,\mu_3+(1-s)\,{\widehat\mu}_3, s\,\delta+(1-s)\,{\widehat\delta}\Big)\nonumber\\
	&= \Psi\Big(s\,(\lambda_1+\lambda_2+\lambda_3)+(1-s)\,({\widehat\lambda}_1+{\widehat\lambda}_2+{\widehat\lambda}_3),\nonumber\\
	&\qquad\quad s\,(\mu_1+\mu_2+\mu_3)+(1-s)\,({\widehat\mu}_1+{\widehat\mu}_2+{\widehat\mu}_3),\: s\,\delta+(1-s)\,{\widehat\delta} \Big)\nonumber\\
	&\leq s\,\Psi(\lambda_1+\lambda_2+\lambda_3,\: \mu_1+\mu_2+\mu_3,\: \delta) + (1-s)\, \Psi({\widehat\lambda}_1+{\widehat\lambda}_2+{\widehat\lambda}_3,\: {\widehat\mu}_1+{\widehat\mu}_2+{\widehat\mu}_3,\: {\widehat\delta})\nonumber\\
	&= s\,\Phi(\lambda_1,\lambda_2,\lambda_3,\mu_1,\mu_2,\mu_3,\delta) + (1-s)\,\Phi({\widehat\lambda}_1,{\widehat\lambda}_2,{\widehat\lambda}_3,{\widehat\mu}_1,{\widehat\mu}_2,{\widehat\mu}_3,{\widehat\delta})\nonumber\\
	&= s\,\Phi(y) + (1-s)\,\Phi({\widehat{y}})\,,
\end{align}
hence $\Phi$ is convex as well.
\end{proof}

\begin{remark}\label{remA2}
We consider an isotropic energy function $W(\lambda_1,\lambda_2,\lambda_3)$ with
\begin{align}
	W(\lambda_1,\lambda_2,\lambda_3) &= \Phi(\lambda_1,\,\lambda_2,\,\lambda_3,\,(\lambda_2\lambda_3),\,(\lambda_1\lambda_3),\,(\lambda_1\lambda_2),\,(\lambda_1\lambda_2\lambda_3))
\label{eq:psiPhiAgreement3}\\
	&= \Psi(\lambda_1+\lambda_2+\lambda_3,\,(\lambda_2\lambda_3)+(\lambda_1\lambda_3)+(\lambda_1\lambda_2),\,(\lambda_1\lambda_2\lambda_3))\nonumber
\end{align}
for all $\lambda_1,\lambda_2,\lambda_3\in\R_+$ with functions $\Psi:\R_+^3\to\R$ and $\Phi:\R_+^7\to\R$. Then the functions $\Psi$ and $\Phi$ do not necessarily fulfil the conditions of the previous lemma, i.e. we do not have an equality  like in \eqref{cond-lem-martin}.

\end{remark}
\begin{proof}
 In order to prove this remark, let us observe that while the equality
\begin{equation}
	\Phi(\lambda_1,\lambda_2,\lambda_3,\mu_1,\mu_2,\mu_3,\delta) = \Psi(\lambda_1+\lambda_2+\lambda_3,\:\mu_1+\mu_2+\mu_3,\:\delta)\label{eq:phiPsiEqualityDefinition}
\end{equation}
holds if $\mu_1=\lambda_2\lambda_3,\:\mu_2=\lambda_1\lambda_3,\:\mu_3=\lambda_1\lambda_2$ and $\delta=\lambda_1\lambda_2\lambda_3$, it generally does \emph{not} hold for arbitrary $\lambda_1,\lambda_2,\lambda_3,\mu_1,\mu_2,\mu_3,\delta\in \R_+$. In particular, we cannot simply apply the proof of the lemma since the equalities  \eqref{eq:representationConvexityProofFails1}  depend on the fact that $\Psi$ and $\Phi$ are equal (in the sense of \eqref{eq:phiPsiEqualityDefinition}) in a point given as a convex combination of two points at which $\Psi$ and $\Phi$ are equal. Since the set
\begin{align}
	\{(\lambda_1,\lambda_2,\lambda_3,\mu_1,\mu_2,\mu_3,\delta)\in\R_+^7 \;|\; \mu_1=\lambda_2\lambda_3,\:\mu_2=\lambda_1\lambda_3,\:\mu_3=\lambda_2\lambda_3,\:\delta=\lambda_1\lambda_2\lambda_3\}
\end{align}
is not convex, we cannot apply this lemma to the general case given by \eqref{eq:psiPhiAgreement3}.

For example, consider the functions
$
	\Psi(a,b,c) \,=\, b^2-2ac\,, \  \Phi(\lambda_1,\lambda_2,\lambda_3,\mu_1,\mu_2,\mu_3,\delta) \,=\, \mu_1^2 + \mu_2^2 + \mu_3^2\,.
$
For $\mu_1=\lambda_2\lambda_3,\,\mu_2=\lambda_1\lambda_3,\,\mu_3=\lambda_2\lambda_3,\,\delta=\lambda_1\lambda_2\lambda_3$ we find
\begin{align*}
	\Psi(\lambda_1+\lambda_2+\lambda_3,&\,\mu_1+\mu_2+\mu_3,\,\delta)= (\mu_1+\mu_2+\mu_3)^2-2\delta(\lambda_1+\lambda_2+\lambda_3)\\
	&= (\lambda_2\lambda_3+\lambda_1\lambda_3+\lambda_2\lambda_3)^2 - 2(\lambda_1\lambda_2\lambda_3)(\lambda_1+\lambda_2+\lambda_3)\\
	&= \lambda_2^2\lambda_3^2 + \lambda_1^2\lambda_3^2 + \lambda_1^2\lambda_2^2 + 2 (\lambda_1^2\lambda_2\lambda_3 + \lambda_1\lambda_2^2\lambda_3 + \lambda_1\lambda_2\lambda_3^2) - 2(\lambda_1\lambda_2\lambda_3)(\lambda_1+\lambda_2+\lambda_3)\\
	&= \lambda_2^2\lambda_3^2 + \lambda_1^2\lambda_3^2 + \lambda_1^2\lambda_2^2 \,=\,\mu_1^2+\mu_2^2+\mu_3^2 \,=\, \Phi(\lambda_1,\lambda_2,\lambda_3,\mu_1,\mu_2,\mu_3,\delta)\,,
\end{align*}
but for $\lambda_i=1,\ \mu_i=2,\ \delta=1$ this equality no longer holds:
\begin{align*}
	\hspace{1cm}\Psi(1+1+1,\,2+2+2,\,1) = \Psi(3,6,1) = 6^2-6 = 30 \neq 12 = 2^2+2^2+2^2 = \Phi(1,1,1,2,2,2,1)\,.\hspace{2cm}\qedhere
\end{align*}
\end{proof}

Combining Lemma \ref{lem-martin} with a polyconvexity criterion given by Ball (see Theorem \ref{critBalleng3}) allows for a simple proof of another criterion proposed  originally by Steigmann (see Theorem \ref{sc1}).
\begin{proposition}
\label{prop:steigmann}  Let  $W:{\rm GL}^+(3)\rightarrow\mathbb{R}$  be an isotropic scalar function and let   $g(\lambda_1,\lambda_2,\lambda_3)$ be a symmetric real-valued function defined on $\mathbb{R}_+^3$ such that
$
W(F)=g(\lambda_1,\lambda_2,\lambda_3)
$
for all $F\in{\rm GL}^+(3)$, where $\lambda_1,\lambda_2,\lambda_3$ are the singular values of $F$.
Let $\psi:\R_+^3\to\R$ be \textbf{convex} and \textbf{nondecreasing} in the first two arguments with
\[
	g(\lambda_1,\lambda_2,\lambda_3)=\psi(\lambda_1+\lambda_2+\lambda_3,\,\lambda_1\lambda_2 + \lambda_1\lambda_3 + \lambda_2\lambda_3,\,\lambda_1\lambda_2\lambda_3)
\]
for all $\lambda_1,\lambda_2,\lambda_3\in\R_+\,$. Then $W$ is polyconvex.
\end{proposition}
\begin{proof}
We define $\Phi:\R_+^7\to\R$,
\[
	\Phi(\lambda_1,\lambda_2,\lambda_3,\mu_1,\mu_2,\mu_3,\delta) = \psi(\lambda_1+\lambda_2+\lambda_3,\,\mu_1+\mu_2+\mu_3,\, \delta)\,.
\]
According to Lemma \ref{lem-martin}, the function $\Phi$ is convex. Furthermore $\Phi$ is invariant under (separate) permutations  of $(\lambda_1,\lambda_2,\lambda_3)$ or $(\mu_1,\mu_2,\mu_3)$, and $\Phi$ is nondecreasing in each $\lambda_i,\mu_i\,$. Then Ball's criterion \ref{critBalleng3} shows that $W$ is polyconvex.
\end{proof}

\begin{remark}
Note carefully that Proposition \ref{prop:steigmann} assumes that the convex function $\psi$ is defined on all  $\R_+^3$, while the domain of $\psi$ is left ambiguous in Steigmann's criterion but clear from the context.
\end{remark}
The following lemma states a necessary condition for polyconvexity.
\begin{lemma}Let  $W:{\rm GL}^+(3)\rightarrow\mathbb{R}$  be an isotropic {polyconvex} function and let   $g(\lambda_1,\lambda_2,\lambda_3)$ be a symmetric real-valued function defined on $\mathbb{R}_+^3$ such that, for all $F\in{\rm GL}^+(3)$
$
W(F)=g(\lambda_1,\lambda_2,\lambda_3),
$
where $\lambda_1,\lambda_2,\lambda_3$ are the singular values of $F$. Then there exists a convex function $\Phi:\R^7\to\Rinf = \R\cup\{+\infty\}$ with
\[
	g(\lambda_1,\lambda_2,\lambda_3) = \Phi(\lambda_1, \;\lambda_2, \;\lambda_3, \;\lambda_1\lambda_2, \;\lambda_1\lambda_3, \;\lambda_2\lambda_3, \;\lambda_1\lambda_2\lambda_3)
\]
for all $\lambda_1,\lambda_2,\lambda_3\in\R_+\,$.
\end{lemma}
\begin{proof}
Since $W$ is polyconvex, there exists a convex function
$
	P:\R^{3\times3}\times\R^{3\times3}\times\R\;\to\;\Rinf
$
with $P(F,\Cof F,\det F) = W(\lambda_1,\lambda_2,\lambda_3)$ for all $F\in{\rm GL}^+(3)$ with eigenvalues $\lambda_1,\lambda_2,\lambda_3\,$. We define
\[
	\Phi(\lambda_1,\lambda_2,\lambda_3,\mu_1,\mu_2,\mu_3,\delta) \;:=\; P\big({\rm diag}(\lambda_1,\lambda_2,\lambda_3),{\rm diag} (\mu_1,\mu_2,\mu_3),\; \delta\big)\,.
\]
Then the convexity of $P$ implies
\begin{align*}
	\Phi(s(\lambda_1,\lambda_2,\lambda_3,\mu_1,&\mu_2,\mu_3,\delta) + (1-s)(\lambdahat_1,\lambdahat_2,\lambdahat_3,\muhat_1,\muhat_2,\muhat_3,\deltahat))\\
	&=P\big(s\cdot{\rm diag}(\lambda_1,\lambda_2,\lambda_3) + (1-s)\cdot{\rm diag}(\widehat{\lambda}_1,\widehat{\lambda}_2,\widehat{\lambda}_3),\\&\qquad\; s\cdot{\rm diag}(\mu_1,\mu_2,\mu_3) + (1-s)\cdot {\rm diag}(\widehat{\mu}_1,\widehat{\mu}_2,\widehat{\mu}_3), s\,\delta + (1-s)\,\deltahat\big)\\
	&\leq s\:P\big({\rm diag}(\lambda_1,\lambda_2,\lambda_3),\; {\rm diag}(\mu_1,\mu_2,\mu_3),\; \delta\big) \:+\: (1-s)\:P\big({\rm diag}(\widehat{\lambda}_1,\widehat{\lambda}_2,\widehat{\lambda}_3),\; {\rm diag}(\widehat{\mu}_1,\widehat{\mu}_2,\widehat{\mu}_3),\; \deltahat\big)\\
	&= s\,\Phi(\lambda_1,\lambda_2,\lambda_3,\mu_1,\mu_2,\mu_3,\delta) + (1-s)\,\Phi(\lambdahat_1,\lambdahat_2,\lambdahat_3,\muhat_1,\muhat_2,\muhat_3,\deltahat)
\end{align*}
for all $s\in[0,1]$ and $\lambda_i,\mu_i,\delta,\lambdahat_i,\muhat_i,\deltahat\in\R_+$, thus $\Phi$ is convex.
Finally we find
\begin{align*}
	\quad\; \Phi(\lambda_1, \;\lambda_2, \;\lambda_3, \;&\lambda_1\lambda_2, \;\lambda_1\lambda_3, \;\lambda_2\lambda_3, \;\lambda_1\lambda_2\lambda_3)= P\big({\rm diag}(\lambda_1,\lambda_2,\lambda_3),\; {\rm diag}(\lambda_1\lambda_2,\lambda_2\lambda_3,\lambda_3\lambda_1),\; \lambda_1\lambda_2\lambda_3\big)\\
	&= P\big({\rm diag}(\lambda_1,\lambda_2,\lambda_3),\, \Cof{\rm diag}(\lambda_1,\lambda_2,\lambda_3),\, \det {\rm diag}(\lambda_1,\lambda_2,\lambda_3)\big) = g(\lambda_1,\lambda_2,\lambda_3)
\end{align*}
for all $\lambda_1,\lambda_2,\lambda_3\in\R_+\,$.
\end{proof}
The next theorem, which we have already mentioned, states sufficient conditions for polyconvexity of functions that have the same form.
\begin{theorem}{\rm(Ball \cite[page 367]{Ball77},  3D sufficient conditions for polyconvexity of isotropic functions)}\label{critBalleng3}\\
Let
$
W(F)=\Phi(\lambda_1,\lambda_2,\lambda_3, \lambda_2\lambda_3,\lambda_3\lambda_1,\lambda_1\lambda_2,\lambda_1\lambda_2\lambda_3),
$
where $\lambda_1,\lambda_2,\lambda_3$ are the singular values of $F\in {\rm GL}^+(3)$, and
\begin{itemize}
\item[a)]  $\Phi:\mathbb{R}_+^7\rightarrow\mathbb{R}$ is convex,
\item[b)] $\Phi(\widetilde{P}\,x,\overline{P}y,\delta)=\Phi(x,y,\delta)$ for all $\widetilde{P},\overline{P}\in \mathcal{P}_3$ (an element $\widetilde{P}$ of $\mathcal{P}_3$, acts
on a  vector  $v\in \R^3$ by permuting its entries) and all $x,y\in\mathbb{R}_+^3$, $\delta\in \R_+$,
\item[c)] $\Phi(x_1,x_2,x_3,y_1,y_2,y_3,\delta)$ is nondecreasing in each $x_i,y_j$, individually.
\end{itemize}
Then $W$ is polyconvex on ${\rm GL}^+(3)$.
\end{theorem}

\newcommand{\eps}{\varepsilon}
\newcommand{\del}{\partial}
\subsection{A direct proof of the positive definiteness of the Hessian matrix $D^2\psi $ in the domain ${D(i_1,i_2)}$}\label{convdir}\setcounter{equation}{0}

In this appendix, we give  a direct proof that the function $\psi$ considered above satisfies the conditions of Proposition \ref{sc2} in the domain ${D(i_1,i_2)}$. These calculations have the disadvantage that a generalization to dimensions three or higher leads to very complicated expressions. On the other hand, they rely on elementary calculus only and are included for the convenience of all readers who prefer this. They may also help to provide an intuition, which expressions arise upon application of this theorem to an energy function and what manipulations might be helpful.

We will use these substitutions that we want to give an overview of at this point:
\begin{gather*}
 R=\sqrt{i_1^2-4i_2},\qquad z=\frac{R}{i_1},\qquad a=\frac{1+z}{1-z},
 a=e^\xi,\qquad s=\sinh\xi,\qquad  c=\cosh\xi,\qquad t=\tanh\xi.
\end{gather*}

\begin{lemma}
 \label{thm:monotone} The function $\psi:{D(i_1,i_2)}\rightarrow\mathbb{R}$, ${\psi}(i_1,i_2)=
\dd e^{\frac{k}{2}\log^2 \frac{i_1+\sqrt{i_1^2-4\,i_2}}{i_1-\sqrt{i_1^2-4\,i_2}}}$
is monotone increasing in $i_1$.
\end{lemma}
{\it Proof.} As long as $i_1>\sqrt{4\,i_2}$,  for each $(i_1,i_2)$, we have
$\frac{\del\psi}{\del i_1}=\frac{2\, k\,\log \frac{i_1+R}{i_1-R}}{R}e^{\frac{k}{2}\log^2\frac{i_1+R}{i_1-R}}\geq 0.$\hfill$\Box$

\begin{lemma}
\label{thm:b}
The inequality $k-\frac{1}{\log^2 a}+\frac{4\,a^2}{(a^2-1)^2}\geq 0$ holds true for all $a\geq 1$ if and only if $k\geq\dd\tel3$. If $k<\dd\tel3$ then
\(
 \lim_{a\searrow 1}\left[k-\frac{1}{\log^2 a}+\frac{4\,a^2}{(a^2-1)^2}\right]<0.
\)
\end{lemma}
\begin{proof}
 Upon the substitution $a=e^\xi$, this expression becomes
 \(
  k-\frac{1}{\log^2 a}+\frac{4\,a^2}{(a^2-1)^2}=k-\tel{\xi^2}+(\tel{\sinh(\xi)})^2\quad  \forall\, \xi\geq 0.
 \)
 We therefore compute (abbreviating $s=\sinh\xi, c=\cosh\xi, t=\tanh\xi$)
 \begin{align*}
  \limxin\left[ -\tel{\xi^2}+(\tel{\sinh(\xi)})^2\right] &=\limxin \frac{\xi^2-s^2}{\xi^2\,s^2}
  =\limxin\frac{2\,\xi-2\,s\,c}{2\,\xi\, s^2+2\,\xi^2\,s\,c}=\limxin\frac{\xi-s\,c}{\xi \,s^2+\xi^2\,s\,c}\\
  &=\limxin\frac{1-c^2-s^2}{s^2+2\,\xi\, s\,c+2\xi \,s\,c+\xi^2\,c^2+\xi^2\,s^2}=\limxin -\frac{2\,s^2}{s^2+4\,\xi \,s\,c+2\,\xi^2\,s^2+\xi^2}\\&
   =\limxin -\frac{2}{1+4\,c\,\frac{\xi}{s}+2\xi^2+\frac{\xi^2}{s^2}}=-\tel3.
 \end{align*}
 We claim that the derivative of this expression $\tel{\sinh^2 \xi}-\tel{\xi^2}$ is positive, i.e.
 \( -2\,s^{-3}\,c+2\,\xi^{-3}\geq 0\)
 or, equivalently,
 \(
  \xi^3\leq s^2\,t.
 \)
 At $\xi=0$, this inequality holds: $0\leq 0$. Hence it is sufficient to compare
 the corresponding derivatives in the same way:
 \(
  3\,\xi^2\leq 2\,s\,c\,t+s^2\,\tel{c^2}=2\,s^2+t^2.
 \)
 Again, differentiating, we obtain
 \(
  6\,\xi\leq 4\,s\,c+2\,t\,\tel{c^2},\) and \(
  3\,\xi\leq 2\,s\,c+\frac{s}{c^3}
 \)
 And again: At $\xi=0$ this is $0\leq 0$ and comparison of the derivatives gives
 \[
  3\leq 2\,c^2+2\,s^2+\frac{c^4-s^2\,3\,c^2}{c^6}
 , \qquad
  3\leq 2\,c^2+2\,c^2-2+\tel{c^2}-3\,\frac{c^2}{c^4}+3\tel{c^4}
, \qquad
 5\leq 4\,c^2-2\tel{c^2}+\frac{3}{c^4}\,.
\]
We need
\(
 0\leq4\,c^6-5\,c^4-2\,c^2+3
\)
for all $\xi\geq 0$ and hence, so to say, for all $c\geq 1$.
Inspection of the polynomial $4\,y^3-5\,y^2-2\,y+3$ shows that it has a local minimum at $y=1$ and
\(
 4\,y^3-5\,y^2-2\,y+3\geq 4-5-2+3=0
\)
holds true for $y\geq 1$.
\end{proof}

\begin{lemma}
\label{thm:unglt}
 The inequality
 \(
  t(a):=\frac{1}{\log a}+k\log a-\tel2(\frac{a-1}{a+1}+\frac{a+1}{a-1})\geq 0
 \)
 holds for all $a\geq 1$ if and only if $k>\tel3$.
\end{lemma}
\begin{proof}
We  can rewrite
 \(
  t(a)=\tel{\log a}+k\log a-\frac{a^2+1}{a^2-1}.
 \)
Using the substitution $\alpha=a^2\searrow 1$, we have
 \begin{align*}
  \limae t(a)&=\limale(\tel{\tel2\log
\al}-\frac{\al+1}{\al-1})=\limale\frac{2\al-2-\al\log \al-\log \al}{(\al-1)\log
\al}\\
  &=\limale\frac{2-\log \al-\frac{\al}{\al}-\tel{\al}}{\log
\al+\frac{\al-1}{\al}}=\limale \frac{1-\tel{\al}-\log \al}{\log \al+1-\tel
\al}=\limae\frac{\tel{\al^2}-\tel{\al}}{\tel \al+\tel{\al^2}}=0.
 \end{align*}
 Moreover,
\(  t'(a)=-\frac{1}{\log^2 a}\tel{a}+k\tel{a}-\frac{2a(a^2-1)-(a^2+1)2a}{(a^2-1)^2}\geq 0\), which is the same as \(k-\tel{\log^2a}+\frac{4a^2}{(a^2-1)^2}\geq 0.\)
Hence, in view of Lemma \ref{thm:b} we have $t'(a)\geq 0$ for all $a\geq 1$ if $k\geq\tel3$. If $k<\tel3$, by Lemma \ref{thm:b}, $\limae t'(a)<0$ which together with $\limae t(a)=0$ implies negativity of $t(a)$ on some interval $(0,\eps)$.
\end{proof}

\begin{lemma}
\label{thm:ungl7}
The inequality
 \(
  -(i_1^2+R^2)\log\frac{i_1+R}{i_1-R}+2\,i_1\,R\left(1+k\log^2\frac{i_1+R}{i_1-R}\right)\geq 0
 \)
 holds for arbitrary $i_1\geq R\geq 0$ if and only if $k\geq \tel3$.
\end{lemma}
\begin{proof}
 After division by $i_1^2$, cancelling of $i_1$ in the arguments of the logarithms and denoting $z=\frac{R}{i_1}$, this is equivalent to
 \[
  -(1+z^2)\log\frac{1+z}{1-z}+2\,z\,\left(1+k\log^2\frac{1+z}{1-z}\right)\geq 0 \qquad \forall\; z\in (0,1)
 \]
 and hence, dividing by $2\,z\,\log\frac{1+z}{1-z}$, to
 \(
\frac{1}{\log\frac{1+z}{1-z}}+k\log\frac{1+z}{1-z}-\tel2(z+\tel{z})=t(\frac{
1+z}{1-z})\geq 0,
 \)
 which (after substitution $a=\frac{1+z}{1-z}$) holds true for all $z\in(0,1)$
by Lemma \ref{thm:unglt} if and only if $k\geq \tel3$.
\end{proof}

\begin{lemma}
 \label{thm:ungl8}
 If and only if $k\geq\tel3$, the following inequality holds for all $i_1\geq 0$, $0\leq i_2\leq\frac{i_1^2}{4}$
 \begin{equation}
 \label{eq:thmungl8}
  2\,i_2\log\frac{i_1+R}{i_1-R}+k \,i_1\log^2\left(\frac{i_1+R}{i_1-R}\right)R+i_1R-i_1^2\log\frac{i_1+R}{i_1-R}\geq 0
 \end{equation}
\end{lemma}
\begin{proof}
 The substitution $R^2=i_1^2-4\, i_2$ implies $i_2=\frac{i_1^2-R^2}{4}$. Hence, the expression \eqref{eq:thmungl8} becomes
 \begin{align*}
  &2\frac{i_1^2-R^2}{4}\log\frac{i_1+R}{i_1-R}+k \,i_1R\log^2\frac{i_1+R}{i_1-R}+i_1R-i_1^2\log\frac{i_1+R}{i_1-R}\geq 0\\
  &\iff-(i_1^2+R^2)\log\frac{i_1+R}{i_1-R}+2\,i_1\,R\left(1+k\log^2\frac{i_1+R}{i_1-R}\right)\geq 0,
 \end{align*}
 which is true by Lemma \ref{thm:ungl7}.
\end{proof}

\begin{lemma}
\label{thm:ungl13}Let $k\geq\tel8$. Then
 \(
  2\,k\, z\log^2\frac{1+z}{1-z}+2\,z-\log\frac{1+z}{1-z}\geq 0,
 \)
 if $z\in (0,1)$.
\end{lemma}
\begin{proof} We have
 $\limzn \left[2\,k\, z\log^2\frac{1+z}{1-z}+2\,z-\log\frac{1+z}{1-z}\right]=0$
 and
 \begin{align*}
  \frac{d}{dz}\left[2\,k\,
z\,\log^2\frac{1+z}{1-z}+2\,z-\log\frac{1+z}{1-z}\right]&=2\,k\,\log^2\frac{1+z}{1-z
}+\frac{8kz}{1-z^2}\log\frac{1+z}{1-z}+2-\frac2{1-z^2}\\
  &=2\,k\,\log^2\frac{1+z}{1-z}+\frac{2z}{1-z^2}(4\,k\,\log\frac{1+z}{1-z}-z).
  \end{align*}
  The derivative is nonnegative, because
  \(
   \limzn \left(4\,k\log\frac{1+z}{1-z}-z\right)\!\!=\!0
  \)
  and
  \(
   \frac{d}{dz}\left(\,{4\,k\,\log\frac{1+z}{1-z}-z}\right)\!=\!\frac{z^2+8\,k-1}{1-z^2}\geq 0.\qedhere
  \)
\end{proof}

\begin{lemma} Let $k>\tel8$. For all $\i_1,i_2\in {D(i_1,i_2)}$, we have
\label{thm:ungl10}
 \(
  2\,k \,R\log^2\frac{i_1+R}{i_1-R}+2\,R-i_1\log\frac{i_1+R}{i_1-R}\geq 0.
 \)
\end{lemma}
\begin{proof}
 Divide by $i_1$ and use $z=\frac{R}{i_1}$ to obtain
 \(
  2\,k\, z\log^2\frac{1+z}{1-z}+2\,z-\log\frac{1+z}{1-z}\geq 0,
 \)
 which is true by Lemma \ref{thm:ungl13}.
\end{proof}
\begin{remark} In the new notations,  for all $(i_1,i_2)\in {D(i_1,i_2)}$, we have
 \begin{align*}
  \frac{\partial^2 \psi}{\partial i_1^2}(i_1,i_2)&=\frac{8\,i_2\,e^{\frac{k}{2}\log^2\frac{i_1+R}{i_1-R}}}
  {\sqrt{i_1^2-4\,i_2}(i_1^2-4\,i_2)(-1)
  (\sqrt{i_1^2-4\,i_2}-i_1)(\sqrt{i_1^2-4\,i_2}+i_1)}\cdot\\
  &\qquad\cdot\left(2\,k\,R\log^2\frac{i_1+R}{i_1-R}
+2\,R -i_1\log\frac{i_1+R}{i_1-R}\right),
 \end{align*}
 and
\begin{align*}
 {\rm det}\,D^2\psi(i_1,i_2)&=128\,k^2 i_2\log\frac{i_1+R}{i_1-R}\tel{(i_1^2-4\,i_2)^2}
 \left(-\tel{(\sqrt{i_1^2-4\,i_2}+i_1)^3(\sqrt{i_1^2-4\,i_2}-i_1)^3}\right)\cdot\\
 &\quad\cdot\left[2\,i_2\log\frac{i_1+R}{i_1-R}+k \,i_1\,R\,\log^2\frac{i_1+R}{i_1-R}
 +i_1\,R-i_1^2\log\frac{i_1+R}{i_1-R}\right]e^{k\,\log^2\frac{i_1+
 \sqrt{i_1^2-4\,i_2}}{i_1-\sqrt{i_1^2-4\,i_2}}}.
\end{align*}
Both above quantities are positive  if and only if $k\geq \frac{1}{3}$ by Lemmas  \ref{thm:ungl10} and \ref{thm:ungl8}.
\end{remark}

\end{document}